\documentclass[11pt]{article}

\def\withcolors{0}
\def\withnotes{0}
\usepackage{ccanonne}

\usepackage{url}            %
\usepackage{booktabs}       %
\usepackage{amsfonts}       %
\usepackage{nicefrac}       %
\usepackage{microtype}      %

\usepackage{amsmath,amsthm,amssymb} %
\usepackage{bbm} 			%
\usepackage{algorithm,algpseudocode} %
\usepackage[top=1in, bottom=1in, left=1.1in, right=1.1in]{geometry} %

\usepackage[style=alphabetic,natbib=true,maxnames=99,maxalphanames=99]{biblatex}
\allowdisplaybreaks

\crefname{assumption}{Assumption}{Assumptions}

\DeclareMathOperator*{\esssup}{ess\,sup}

\title{Optimal Rates for Nonparametric Density Estimation under Communication Constraints}
\author{Jayadev Acharya\thanks{Cornell University. Email: \email{acharya@cornell.edu}}
 \and Cl\'ement L. Canonne\thanks{University of Sydney. Email: \email{clement.canonne@sydney.edu.au}}
 \and Aditya Vikram Singh\thanks{Indian Institute of Science, Bangalore. Email: \email{adityavs@iisc.ac.in}}
 \and Himanshu Tyagi\thanks{Indian Institute of Science, Bangalore. Email: \email{htyagi@iisc.ac.in}}}

\begin{document}
	\maketitle
	
\begin{abstract}
  We consider density estimation for Besov spaces when each sample is quantized to only a limited number of bits.
We provide a noninteractive adaptive estimator that exploits the sparsity of wavelet bases, along with a simulate-and-infer technique from parametric estimation under communication constraints. We show that our estimator is nearly rate-optimal by deriving minimax lower bounds that hold even when interactive protocols are allowed. Interestingly, while our wavelet-based estimator is almost rate-optimal for Sobolev spaces as well, it is unclear whether the standard Fourier basis, which arise naturally for those spaces, can be used to achieve the same performance.
\end{abstract}
\clearpage
	\tableofcontents
	\newpage
	\section{Introduction} \label{sec:introduction}
	  Estimating distributions from samples is a fundamental statistical task. 
Modern applications, such as those arising in Federated Learning or the Internet of Things (IoT), often limit access to the true data samples.
\newer{One common limitation} in large scale distributed systems is communication constraints, which require that each data sample must be {compressed} to a small number of bits. \medskip

Most prior work on \newer{communication-constrained estimation} has focused on parametric  problems such as Gaussian mean estimation and discrete distribution estimation.
In this work, we study nonparametric density estimation under communication constraints where independent samples from an unknown distribution (whose density $f$ lies in a suitable function class) are distributed across players (one sample per player), and each player can only send $\numbits$ bits about their sample to a central referee; the referee outputs an estimate of $f$ based on these $\numbits$-bit messages. 
This problem has been considered by~\citep{BarnesHO20} for densities in H\"older classes, a relatively simple class for which the normalized histogram with uniform bins is known to be optimal in the centralized setting. For densities
in H\"older classes, a natural method for density estimation in the distributed setting is to quantize each data sample into uniform bins and use the optimal estimator for distributed discrete distribution estimation. Indeed,~\cite{BarnesHO20} shows that this is optimal for distributed estimation of densities from the H\"older class under communication constraints. However, this simple estimator does not seem to extend to the richer Sobolev class and the most general Besov classes. 
In particular, the following question is largely open:
\begin{center}\itshape
	How to quantize samples to estimate densities from Besov classes under\linebreak[2] communication constraints?
\end{center}
We resolve this question for the cases  when the density belongs to a Besov class with known parameters (\emph{nonadaptive} setting), as well as when the density belongs to a Besov class where only upper and lower bounds on parameters are known (\emph{adaptive} setting). Specifically, our proposed estimators exploit the sparsity of wavelet basis for the Besov class, and use vector quantization followed by the distributed simulation technique introduced in \citep{AcharyaCT20inf} for distributed parametric estimation.
We also establish information-theoretic lower bounds that prove the optimality of our estimators (up to logarithmic terms in the adaptive setting). 

\subsection{Problem setup} \label{subsec:setup}
Let $X_1,\ldots,X_n$ be i.i.d. samples from an unknown distribution with density $f$ supported on $\cX \eqdef [0,1]$ and belonging to the Besov space $\cB(p,q,s)$. There are $n$ distributed users (\emph{players}) and player $i$ has access to sample $X_i$. Each player can only transmit an $\numbits$ bit message to a central server (\emph{referee}) whose goal, upon observing $\numbits$-bit messages from $n$ players, is to estimate~$f$.
We consider an \emph{interactive} setting, where the current player observes the messages from previous players and can use them to design their message.\footnote{Since our lower bounds rely on general results in~\cite{AcharyaCT20}, we borrow the notation from that paper.}
That is, in round $i$, player $i$ chooses a communication-constrained channel (randomized mapping) $W_i\colon\cX\to\{0,1\}^\numbits$ as a function of prior messages $Y_1, \ldots, Y_{i-1}$ and randomness $U$ available to all players; she then passes $X_i$ through $W_i$ to generate $Y_i\in\{0,1\}^\numbits$.
The referee observes the messages $Y_1,\ldots, Y_n$ and outputs an estimate $\hat f$ of $f$. 
We term such an $\hat f$ an \emph{$(\ns, \numbits)$-estimate}.
Let
$\cE_{\ns, \numbits}$ denote the set of all $(\ns, \numbits)$-estimates. 

Our goal is to design estimators that achieve the minimax expected $\cL_r$ loss defined, for $r\geq 1$, by
\begin{equation}\label{eq:def:minmax:risk}
\cL_r^*(\ns, \numbits, p,q,s)\eqdef \inf_{\hat f \in \cE_{\ns, \numbits}}\sup_{f \in \cB(p,q,s)}
\bE{f}{\big\| \hat f-f \big\|_r^r}.
\end{equation}
For upper bounds on $\cL_r^*(\ns, \numbits, p,q,s)$, we design algorithms under the more restricted \emph{noninteractive} protocols, where the channel $W_i$ of player $i$ is not allowed to depend on the messages $Y_1, \ldots, Y_{i-1}$ or on the common randomness $U$, but may depend on the private randomness $U_i$ available at player $i$, where $U_1,\dots,U_n$ are independent of each other and jointly of $U$. Noninteractive protocols are easier to implement, and result in simpler engineering for the distributed system. 

\subsection{Our results and techniques}
Our first result is an information-theoretic lower bound on $\cL_r^*$.
\begin{theorem}
  \label{theo:lb}
For any $p,q,s,r$, there exist constants $C=C(p,q,s,r)>0, \alpha=\alpha(p,q,s,r)>0$ such that 
	\begin{align*}
        \lefteqn{\cL_r^*(\ns, \numbits, p,q,s)}
        \\
        &\geq C\cdot
	\begin{cases}
	    \max\{\ns^{-\frac{rs}{2s+1}}, (\ns 2^\numbits)^{-\frac{rs}{2s+2}}\}, &r \leq (s+1)p,\\
	    \max\{\ns^{-\frac{rs}{2s+1}}, (\ns 2^\numbits)^{-\frac{r(s-1/p+1/r)}{2(s-1/p)+2}}\log(\ns 2^\numbits)^{-\alpha}\}, &r \in \paren{(s+1)p,(2s+1)p}, \\
	    \max\{\big(\frac{\ns}{\log \ns}\big)^{-\frac{r(s-1/p+1/r)}{2(s-1/p)+1}}, (\ns 2^\numbits)^{-\frac{r(s-1/p+1/r)}{2(s-1/p)+2}}\log(\ns 2^\numbits)^{-\alpha}\}, &r\geq(2s+1)p.
	\end{cases}
	\end{align*}
\end{theorem}
We emphasize that this lower bound applies to \emph{interactive} protocols as defined in \cref{subsec:setup}. 
When the parameters $p,q,s$ of the Besov space are known, we design a \emph{noninteractive} estimator that achieves the optimal rate for $r\leq p$.
\begin{theorem}
  \label{theo:ub:nonadaptive}
For any $r\geq 1$ and $p,q,s$ with $r\leq p$, there exist a constant $C=C(p,q,s,r)$ and an $(\ns, \numbits)$-estimate $\hat{f}$ formed using a noninteractive protocol such that
	\begin{align*}
	\sup_{f \in \cB(p,q,s)} \bE{f}{\big\| \hat f-f \big\|_r^r}
	 \leq 
	C\max\{\ns^{-\frac{rs}{2s+1}}, (\ns 2^\numbits)^{-\frac{rs}{2s+2}}\}.
	\end{align*}
\end{theorem}
We finally design an \emph{adaptive}, \emph{noninteractive} estimator that only requires bounds on $s$, and no further knowledge of $p$ and $q$. 
 Moreover, this estimator achieves (up to logarithmic factors) the optimal rate for all parameter values.
\begin{theorem}
  \label{theo:ub:adaptive}
For any $N\in \N$, $r\geq 1$, and $p,q,s$ with $1/p < s < N$, there exist constants $C=C(p,q,s,r),\alpha=\alpha(p,q,s,r)$, and an $(\ns, \numbits)$-estimate $\hat{f}$ formed using a noninteractive protocol
such that 
	\begin{align*}
	\sup_{f \in \cB(p,q,s)} \bE{f}{\big\| \hat f-f \big\|_r^r}
	 \leq C \log^\alpha n \cdot
	\begin{cases}
	    \max\{\ns^{-\frac{rs}{2s+1}}, (\ns 2^\numbits)^{-\frac{rs}{2s+2}}\}, &r\leq(s+1)p,\\
	    \max\{\ns^{-\frac{rs}{2s+1}}, (\ns 2^\numbits)^{-\frac{r(s-1/p+1/r)}{2(s-1/p)+2}}\}, &r \in \paren{(s+1)p,(2s+1)p},\\
	    \max\{\ns^{-\frac{r(s-1/p+1/r)}{2(s-1/p)+1}}, (\ns 2^\numbits)^{-\frac{r(s-1/p+1/r)}{2(s-1/p)+2}}\}, &r\geq(2s+1)p,
	\end{cases}
	\end{align*}
	where the protocol only requires knowledge of $N$ (an upper bound on $s$).
\end{theorem}

In summary, for all $r\ge 1$, the minimax $\cL_r$ loss of estimating $\cB(p,q,s)$ (up to logarithmic factors) is
\begin{equation}
  \label{eq:summary:rate:results}
\cL_r^*(\ns, \numbits, p,q,s) \asymp
\begin{cases}
	    \max\{\ns^{-\frac{rs}{2s+1}}, (\ns 2^\numbits)^{-\frac{rs}{2s+2}}\}, &r\leq(s+1)p,\\
	    \max\{\ns^{-\frac{rs}{2s+1}}, (\ns 2^\numbits)^{-\frac{r(s-1/p+1/r)}{2(s-1/p)+2}}\}, &r \in \paren{(s+1)p,(2s+1)p},\\
	    \max\{\ns^{-\frac{r(s-1/p+1/r)}{2(s-1/p)+1}}, (\ns 2^\numbits)^{-\frac{r(s-1/p+1/r)}{2(s-1/p)+2}}\}, &r\geq(2s+1)p.
	\end{cases}
\end{equation}
It it worth noting that the first term of the maximum in all cases is the standard, unconstrained nonparametric rate (\emph{cf.}~\cite{DonohoJKP96}, or the discussion below), while the second term reflects the convergence slowdown due to the communication constraints. The effect of communication constraints disappears when $\numbits$ is sufficiently large. In particular, we get back the centralized rates when $\numbits$ satisfies
\[
\numbits \geq 
\begin{cases}
	\Paren{\frac{1}{2s+1}}\log\ns, &\text{for } r\leq(s+1)p, \\
	\Paren{\frac{2s(1-1/r)}{(2s+1)(s-1/p+1/r)} - \frac{1}{2s+1}}\log\ns, &\text{for } r \in \paren{(s+1)p,(2s+1)p},\\
	\Paren{\frac{1}{2(s-1/p)+1}}\log\ns, &\text{for } r\geq(2s+1)p.
\end{cases}
\]

For the standard $\cL_2$ loss, with, say $p\geq 2$, the minimax rate becomes the more interpretable quantity
\[
\cL_2^*(\ns, \numbits, p,q,s) \asymp
	    \max\{\ns^{-\frac{2s}{2s+1}}, (\ns 2^\numbits)^{-\frac{2s}{2s+2}}\},
\]
where we see that the $\numbits$-bit communication constraint reduces the exponent of the convergence rate from $\frac{2s}{2s+1}$ to $\frac{2s}{2s+2}$. As $\numbits$ grows or $s$ tends to $\infty$, the two rates coincide. Finally, from~\eqref{eq:summary:rate:results} we observe qualitative changes at $r=(s+1)p$ and $r=(2s+1)p$, where the rate exponent changes slope. This phenomenon, whose analogue is observed in the unconstrained setting~\cite{DonohoJKP96} as well as under local privacy constraints~\cite{ButuceaDKS20}, is sometimes referred to as the \emph{elbow effect}.

\paragraph{Quantize, simulate and infer.}
A conceptually simple technique for distributed estimation under communication constraints (``simulate-and-infer'') was proposed in~\cite{AcharyaCT20inf}, which uses communication to \emph{simulate} samples from the unknown distribution, and provides an optimal rate estimator for \emph{discrete} distribution estimation under communication constraints. A natural extension of this approach for nonparametric estimation would be to quantize the samples to the available number of bits (\ie $\numbits$) and use this quantized sample to estimate the distribution. However, it is unclear if this approach gives optimal rates. Instead, in our approach, we quantize without inducing any ``loss of information.'' Specifically, we form an approximately sufficient statistic (based on wavelets) that can be represented using a finite number of bits and does not result in rate loss. The number of bits could still be more than $\numbits$,  and therefore, we then use simulate-and-infer to generate samples from the statistic. Thus, the loss of information due to communication constraints only happens in the last step, when we use multiple samples to simulate a sample from the sufficient statistic; the quantization part is just for efficient finite representation.

\paragraph{Sobolev spaces and Fourier bases.}
A first approach that we tried for Sobolev spaces was to use the Fourier basis, a natural choice for a Sobolev space. However, all our attempts led to a suboptimal performance either in the dependence on $\numbits$ (we were not able to get the exponential $2^{-\numbits}$ dependence) or the exponent of $\ns$ (\ie when we tried to get a $2^{-\numbits}$ dependence, this resulted in a suboptimal exponent of $\ns$). Somewhat surprisingly, the more general wavelet-based approach described above gives us tight bounds for Sobolev space as well, since Sobolev
space $\cS(\beta)=\cB(2,2,\beta)$. Thus, it seems necessary to use the wavelet representation even for Sobolev spaces to get an appropriately ``small'' statistic~--~sparsity of wavelets is very useful for inference under communication constraints.

\subsection{Prior work}
The seminal work of Donoho, Johnstone, Kerkyacharian, and Picard~\cite{DonohoJKP96} proposed wavelet estimators for Besov class in the centralized setting, and showed that these estimators achieve near-optimal rates of convergence (up to logarithmic factors),
\begin{equation}
  \label{eq:summary:rate:results:centralized}
\cL_r^*(\ns, \infty, p,q,s) \asymp
\begin{cases}
	    n^{-\frac{rs}{2s+1}}, &r<(2s+1)p,\\
	    n^{-\frac{r(s-1/p+1/r)}{2(s-1/p)+1}}, &r\geq(2s+1)p.
	\end{cases}
\end{equation}
Their results highlight the fact that linear estimators are inherently suboptimal for estimation with respect to $\cL_r$ losses, when $r$ is large; that is, some nonlinearity in the estimator is required to achieve optimal rates for $r>p$. In particular, they show that nonlinearity in the form of \emph{thresholding} achieves optimal rates for $r>p$ (see \cref{sec:preliminaries:centralized} below for details). Further, they use thresholding to design \emph{adaptive} estimators that achieve near-optimal rates.
These minimax rates exhibit the aforementioned \emph{elbow effect}, where the error exponent is only piecewise linear, and changes slope at $r=(2s+1)p$. We refer the reader to~\cite{DonohoJKP96} for a further discussion of these phenomena. \medskip

Butucea, Dubois, Kroll, and Saumard~\cite{ButuceaDKS20} recently extended these ideas to obtain near-rate optimal estimators for Besov spaces under local differential privacy constraints. Their adaptive estimator, as well as the information-theoretic lower bounds they establish, show that similar phenomena occur in the context of locally private nonparametric estimation. Our work, specifically the analysis of our adaptive estimator, draws upon some of the ideas of~\cite{ButuceaDKS20}, with some crucial differences. In particular, the key ideas underlying our estimators~--~the wavelet-induced sparsity (\cref{clm:sparsity}), the use of distributed simulation, and vector quantization~--~are neither present in nor applicable to the setting of~\cite{ButuceaDKS20} (where the introduction of random noise to ensure differential privacy effectively removes wavelet sparsity). Furthermore, our lower bounds even apply to interactive protocols, unlike the lower bounds from~\cite{ButuceaDKS20} which are restricted to the noninteractive setting. \medskip

In summary, our paper is the first to derive the counterpart of the nonparametric density estimation results of~\cite{DonohoJKP96,ButuceaDKS20} under communication constraints, and shows that the analogue of the phenomena observed in~\cite{DonohoJKP96} holds in the communication-constrained setting.

\paragraph{Other works on distributed estimation.} We briefly discuss the related literature on distributed (communication-constrained) estimation problems. 
\cite{ZhuLafferty18,SzaboZanten20a,SzaboZanten20b,CaiWei21} have studied the problem of distributed nonparametric function estimation (regression) under a Gaussian white noise model in a noninteractive setting with $\ns$ players, where each player observes an independent copy of the stochastic process $\mathrm{d}Y(t) = f(t)\mathrm{d}t + (1/\sqrt{N}) \mathrm{d}W(t)$, $0\leq t \leq 1$. Here $W(t)$ is the standard Wiener process, and $f$ is the function to be estimated.
\cite{ZhuLafferty18} derive minimax rates for $f$ in Sobolev space under $\cL_2$ loss, where each player can send at most $\numbits$ bits.
\cite{SzaboZanten20a} derive minimax rates for $f$ in the Besov space $\cB(2,\infty,s)$ (``Sobolev type'') under $\cL_2$ loss, and $f$ in $\cB(\infty,\infty,s)$ (``H\"older type'') under $\cL_\infty$ loss, where each player can send at most $\numbits$ (assumed to be at least $\log N$) bits on average. Further, the paper proposes  near-optimal adaptive estimators (based on Lepski's method) that adapt to the smoothness parameter $s$, provided that $s \in [s_{\min},s_{\max})$, where $s_{\min}$ depends on $\ns,\numbits$ and $s_{\max}$ can be arbitrary.
\cite{SzaboZanten20b} further study the problem of adaptivity for $\cB(2,\infty,s)$ under $\cL_2$ loss and $\cB(\infty,\infty,s)$ under $\cL_\infty$ loss, and answers the question of whether it is possible to design adaptive estimators (adapting to the smoothness parameter $s$) that attain centralized minimax rates while {also} having the expected communication budget nearly the same as that of a minimax optimal distributed estimator that knows $s$. The paper shows that this is possible for $\cB(2,\infty,s)$ under $\cL_2$ loss provided that $s$ is below a certain threshold, and is impossible for $\cB(\infty,\infty,s)$ under $\cL_\infty$ loss. 
\cite{CaiWei21} derive minimax rates for $f$ in Besov spaces $\cB(p,q,s)$ (for $p\geq 2$) under $\cL_2$ loss, where each player can send at most $\numbits$ bits on average. In addition, it studies the problem of adaptivity and characterizes the minimax communication budget of adaptive estimators (adapting to parameters $p\geq 2, q > 1, s > 0$) that achieve centralized rates. The adaptive estimator proposed in this paper is based on thresholding, where the thresholding is done locally by each player. 

Interestingly, when the communication budget $\numbits$ is insufficient to achieve centralized rates, the minimax rates in these distributed nonparametric \emph{function estimation} problems decay polynomially in $\numbits$, which is in contrast to the minimax rates we obtain for distributed nonparametric \emph{density estimation} problem, where the decay is exponential in $\numbits$. \medskip

We now discuss related works on distributed parametric estimation problems. ~\cite{BarnesHO20} establish lower bounds on parametric density estimation, and on some restricted nonparametric families (H\"older classes) by bounding the Fisher information.~\cite{AcharyaCT20inf} obtain upper and lower bounds for discrete distribution estimation; our algorithms leverage the concept of \emph{distributed simulation} (``simulate-and-infer'') introduced in that context.~\cite{HanOW18} derive lower bounds for various parametric estimation tasks, including discrete distributions and continuous (parametric) families such as high-dimensional Gaussians (including the sparse case).~\cite{AcharyaCT20}, building on~\cite{AcharyaCLST20} (which focused on learning and testing discrete distributions), developed a general technique to prove estimation lower bounds for parametric families; our lower bounds rely on their framework, by suitably extending it to handle the nonparametric case. \medskip

We note that there are other approaches for establishing lower bounds under communication constraints such as the early works~\cite{ZDJW:13},~\cite{Shamir:14}; ~\cite{GargMN14, BGMNW:16}, where
bounds for specific inference problems under communication constraints were obtained;
and~\cite{BHO:19}, where Cram\'er--Rao bounds for this setting were developed.
We found the general approach of~\cite{AcharyaCT20} best fits our specific application, where we needed to handle interactive communication as well as a nonuniform prior on the parameter in the lower bound construction.

	\section{Preliminaries} \label{sec:besov}
		In this section, we first set out the notation used in the paper, before formally defining Besov spaces. Then, we list the assumptions that we make on the density function $f$, and briefly recall some aspects the existing estimators for density estimation in the centralized setting.  \medskip

Given two integers $m\leq n$, we write $\ibrac{m,n}$ for the set $\{m,m+1,\dots, n\}$ and $\ibrac{n}$ for $\ibrac{1,n}$. For two sequences or functions $(a_n)_n,(b_n)_n$, we write $a_n \lesssim b_n$ if there exists a constant $C>0$ (independent of $n$) such that $a_n \leq Cb_n$ for all $n$, and $a_n \asymp b_n$ if both $a_n \lesssim b_n$ and $a_n \gtrsim b_n$. For a function $g$, $\supp{g}$ denotes the support of $g$.

\subsection{Besov spaces}
Our exposition here is based on \cite{DonohoJKP96, HardleKPT12}. We start with a discussion on wavelets.

\paragraph{Wavelets.} A wavelet basis for $\cL^2(\R)$ is generated using two functions: $\phi$ (father wavelet) and $\psi$ (mother wavelet). The main feature that distinguishes a wavelet basis from the Fourier basis is that the functions $\phi$ and $\psi$ can have compact support. More precisely, there exists a function $\phi\colon \R \to \R$ such that
\begin{enumerate}
	\item\label{cond1} $\set{\phi(\cdot - k) : k \in \Z}$ forms an orthonormal family of $\cL^2(\R)$. Let $V_0 = \operatorname*{span}\set{\phi(\cdot - k) : k \in \Z}$.
	\item\label{cond2} For $j \in \Z$, let $V_j = \operatorname*{span}\set{\phi_{j,k} : k \in \Z}$, where $\phi_{j,k}(x) = 2^{j/2} \phi(2^j x - k)$. Then $V_j \subseteq V_{j+1}$.
	\item\label{cond3} $\phi \in \cL^2(\R)$, $\int \phi(x) dx = 1$. 
	\item[] We note that~\eqref{cond1},~\eqref{cond2}, and~\eqref{cond3} together ensure that $\cap_{j \in \Z} V_j = \set{0}$ and $\cup_{j \in \Z} V_j = \cL^2(\R)$.
	\item\label{cond4} $\phi$ satisfies the following regularity conditions for a given $N \in \Z_+$:
	\begin{enumerate}
		\item There exists a bounded non-increasing function $\Phi$ such that $\int \Phi(\abs{x}) \abs{x}^N dx < \infty$, and $\abs{\phi(x)} \leq \Phi(\abs{x})$ almost everywhere.
		\item $\phi$ is $N+1$ times (weakly) differentiable and $\phi^{(N+1)}$ satisfies $\esssup_x \sum_{k \in \Z} \abs{\phi^{(N+1)}(x-k)} < \infty$. 
	\end{enumerate}
	Any $\phi$ satisfying (a) and (b) is said to be $N$-regular.
\end{enumerate}
Let $W_j \subseteq \cL^2(\R)$ be a subspace such that $V_{j+1} = V_j \bigoplus W_j$ (i.e. $V_{j+1} = V_j + W_j$ and $V_j \cap W_j = \set{0}$). Then, there exists a function $\psi\colon \R \to \R$ such that
\begin{enumerate}
	\item $\set{\psi(\cdot - k) : k \in \Z}$ forms an orthonormal basis of $W_0$.
	\item $\operatorname*{span}\set{\psi_{j,k} : j \in \Z, k \in Z} = \cL^2(\R)$, where $\psi_{j,k}(x) = 2^{j/2} \psi(2^j x - k)$.
	\item $\psi$ satisfies the same regularity conditions as $\phi$.
\end{enumerate}
For any $\low \in \Z$, we can decompose $\cL^2(\R)$ as
\[
\cL^2(\R) = V_{\low} \bigoplus W_{\low} \bigoplus W_{\low+1} \bigoplus \cdots .
\]
That is, for any $f \in \cL^2(\R)$
\begin{equation} \label{eqn:waveletexp}
f = \sum_{k \in \Z} \alpha_{\low,k} \phi_{\low,k} + \sum_{j \geq \low} \sum_{k \in \Z} \beta_{j,k} \psi_{j,k},
\end{equation}
where
\[
\alpha_{\low,k} = \int f(x) \phi_{\low,k}(x) dx, \quad \beta_{j,k} = \int f(x) \psi_{j,k}(x) dx
\]
are called the wavelet coefficients of $f$. Moreover, for a father wavelet $\phi$, there is a canonical mother wavelet $\psi$ (Section 5.2 in \cite{HardleKPT12}) corresponding to $\phi$.

\paragraph{Besov spaces.} We now define the Besov space $\cB(p,q,s)$ with parameters $p,q,s$, where $1\leq p,q \leq \infty$, $s > 0$. Let $\phi$ be a father wavelet satisfying properties \ref{cond1}--\ref{cond4} above, with $N > s$. Then,
\begin{equation} \label{eqn:defbesov}
	f \in \cB(p,q,s) \iff \norm{\alpha_{0\cdot}}_p +
	\paren{\sum_{j=0}^{\infty} \paren{   2^{s+\frac{1}{2}-\frac{1}{p}} \norm{\beta_{j\cdot}}_p   }^q}^{1/q} < \infty
\end{equation}
where $\norm{\alpha_{0\cdot}}_p$ and $\norm{\beta_{j\cdot}}_p$ are the $\ell_p$ norms of the sequences $\set{\alpha_{0,k}}_{k \in \Z}$ and  $\set{\beta_{j,k}}_{k \in \Z}$ respectively. The sequences $\set{\alpha_{0,k}}_{k \in \Z}, \set{\beta_{j,k}}_{k \in \Z}$ come from the wavelet expansion of $f$ using the father wavelet $\phi$ and the corresponding mother wavelet $\psi$. The definition \eqref{eqn:defbesov} of $\cB(p,q,s)$ is invariant to the choice of $\phi$ as long as $N > s$. For the purposes of defining Besov norm, we fix a particular $\phi,\psi$, where $\phi$ is $N$-regular with $N >s$. Then, the \textit{Besov norm} of a function $f$ is defined as
\begin{equation}
	\label{eq:besov:norm}
	\norm{f}_{pqs} := \norm{\alpha_{0\cdot}}_p +
	\paren{\sum_{j=0}^{\infty} \paren{   2^{s+\frac{1}{2}-\frac{1}{p}} \norm{\beta_{j\cdot}}_p   }^q}^{1/q}.
\end{equation}

\subsection{Assumptions}
We make the following assumptions on the density $f$:
\begin{enumerate}
	\item $f$ is compactly supported: without loss of generality, $\supp{f} \subseteq [0,1]$.
	\item Besov norm of $f$ is bounded: without loss of generality,  $\norm{f}_{pqs} \leq 1$.
\end{enumerate}
Our algorithm works with any father and mother wavelets $\phi$ and $\psi$
satisfying the following conditions:
\begin{enumerate}
	\item $\phi$ and $\psi$ are $N$-regular, where $N > s$, and
	\item $\supp{\phi},\supp{\psi} \subseteq [-A,A]$ for some integer $A > 0$ (which may depend on $N$).
\end{enumerate}	
As an example, Daubechies' family of wavelets \cite{Daubechies92} satisfies these assumptions.

\subsection{Density estimation in centralized setting}
  \label{sec:preliminaries:centralized}
In the centralized setting, $X_1, \ldots, X_n$ from an unknown density $f\in\cB(p,q,s)$ are accessible to the estimator. Let the wavelet expansion of $f$ be
\begin{equation}
	\label{eq:wavelet:expansion}
	f = \sum_{k \in \Z} \alpha_{0,k} \phi_{0,k} + \sum_{j \geq 0} \sum_{k \in \Z} \beta_{j,k} \psi_{j,k},
\end{equation}
where $\phi_{j,k}(x) = 2^{j/2}\phi(2^j x - k)$, $\psi_{j,k}(x) = 2^{j/2}\psi(2^j x - k)$.
For any $\low,\high \in \Z$ with $\high \geq \low$, we have (see~\cref{fact:waveletlevel})
$
\sum_{k \in \Z} \alpha_{\low,k} \phi_{\low,k} + \sum_{j = \low}^{\high-1} \sum_{k \in \Z} \beta_{j,k} \psi_{j,k} = \sum_{k \in \Z} \alpha_{\high,k} \phi_{\high,k}.
$
Note that for a given $j,k$, $\hat{\alpha}_{j,k} \eqdef \frac{1}{n} \sum_{i=1}^{n}\phi_{j,k}(X_i)$ is an unbiased estimate of $\alpha_{j,k}$. Thus, for some $\high \in \Z_+$, an estimate of $f$ is
\begin{equation} \label{eq:central:nonadapt}
\hat{f}_{\rm lin} = \sum_{k \in \Z} \hat{\alpha}_{\high,k} \phi_{\high,k},
\quad \hat{\alpha}_{\high,k} \eqdef \frac{1}{n} \sum_{i=1}^{n}\phi_{\high,k}(X_i)
\end{equation}
where $\high$ is chosen depending on $n$ and parameters $p,q,s$ to minimize the (worst-case) $\cL_r$ loss. This simple estimator (with appropriate choice of $\high$) is rate-optimal when $1 \leq r \leq p$, but is sub-optimal when $r > p$ \cite{DonohoJKP96}. Moreover, setting $\high$ requires knowing the Besov parameters $p,q,s$, which renders this estimator nonadaptive. The main contribution of \cite{DonohoJKP96} was to demonstrate that \emph{thresholding} leads to estimators that are (i)~near-optimal for every $r \geq 1$; (ii)~adaptive, in the sense that the estimator does not use the values of parameters $p,q,s$ as long as $s$ lies in a certain range. For a given $\low,\high \in \Z_+$, $\low \leq \high$, a thresholded estimator outputs the estimate
\begin{equation} \label{eq:central:adapt}
\hat{f}_{\rm thresh} = \sum_{k \in \Z} \hat{\alpha}_{L,k} \phi_{L,k} + \sum_{j = \low}^{\high} \sum_{k \in \Z} \tilde{\beta}_{j,k} \psi_{j,k},
\quad \tilde{\beta}_{j,k} = \hat{\beta}_{j,k} \indic{|\widehat{\beta}_{j,k}| \geq t_j}
\end{equation}
where $\hat{\alpha}_{\low,k} \eqdef \frac{1}{n} \sum_{i=1}^{n}\phi_{\low,k}(X_i)$, $\hat{\beta}_{j,k} = \frac{1}{n} \sum_{i=1}^{n} \psi_{j,k}(X_i)$, and $t_j$ is a fixed threshold proportional to $\sqrt{j/n}$; here, $\low,\high$ depend on $n$, but not on parameters $p,q,s$.
Our proposed estimators draw upon these classical estimators.

	\section{Algorithms} \label{sec:algo}
	  We propose algorithms for density estimation under communication constraints that achieve optimal/near-optimal performance in terms of $n$ (number of players) and $\numbits$ (number of bits each player can send). Designing a density estimator in the communication-constrained setting consists of: (i) specifying the sample-dependent $\numbits$-bit message that a player sends to the referee; (ii) specifying the density estimate that the referee outputs based on the $\numbits$-bit messages from the $n$ players. As in the unconstrained setting, we estimate $f$ by estimating its wavelet coefficients. 

Our estimators consist of three ingredients: \emph{wavelet-induced sparsity}, \emph{vector quantization}, and \emph{distributed simulation}. We describe these next.

\paragraph{(i) Wavelet-induced sparsity.} Let the wavelet expansion of the density function $f$ be given by~\eqref{eq:wavelet:expansion}. 
For a given $J \in \Z_+$, partition the interval $[0,1]$ into $2^J$ uniform bins as
\begin{equation} \label{eqn:bins}
[0,1] = \bigcup_{t=0}^{2^{J}-1} E_t^{(J)} \ \text{ where } \
E_t^{(J)} \eqdef
\begin{cases}
	\left[ t{2^{-J}},(t+1){2^{-J}} \right) & \text{if } t \in \ibrac{0,2^{J}-2}, \\
	\left[ 1 - {2^{-J}}, 1 \right] & \text{if } t = 2^{J}-1.
\end{cases}
\end{equation}
For a bin $E_t^{(J)}$, $t \in \ibrac{0,2^{J}-1}$, let
\begin{align}
\cA_t^{(J)} &\eqdef \set{ k \in \Z : E_t^{(J)} \cap \: {\rm supp}(\phi_{J,k}) \text{ is non-empty} }; \label{eqn:At}\\
\cB_t^{(J)} &\eqdef \set{ k \in \Z : E_t^{(J)} \cap \: {\rm supp}(\psi_{J,k}) \text{ is non-empty} }. \label{eqn:Bt}
\end{align}
That is, for $x \in E_t^{(J)}$, we have
$\phi_{J,k}(x) = 0 \text{ for } k \notin \cA_t^{(J)}$, and $\psi_{J,k}(x) = 0 \text{ for } k \notin \cB_t^{(J)}$.
By ``wavelet-induced sparsity,'' we mean the following:
\begin{claim} \label{clm:sparsity}
	Let $[0,1] = \bigcup_{t=0}^{2^{J}-1} E_t^{(J)}$ as in~\eqref{eqn:bins}. Then, for each $t \in \ibrac{0,2^{J}-1}$,
	\[
	| \cA_t^{(J)} | \leq 2(A+2),\; \ |\cB_t^{(J)}| \leq 2(A+2),
	\]
where $A$ is the assumed bound for points in the support of $\phi$ and $\psi$.        
\end{claim}
The claim follows from the observation that $\phi_{J,k}$ (resp., $\psi_{J,k}$) is obtained by translating $\phi_{J,0}$ (resp., $\psi_{J,0}$) in steps of size $2^{-J}$, and that ${\rm supp}(\phi_{J,k}) \subseteq [-A2^{-j},A2^{-j}]$  (resp., ${\rm supp}(\psi_{J,k}) \subseteq [-A2^{-j},A2^{-j}]$).

\paragraph{(ii) Vector quantization.} Consider the problem of designing a randomized algorithm that takes as input an arbitrary $x \in \R^d$ satisfying $\norm{x}_\infty \leq B$ and outputs a random vector $Q(x) \in \R^d$ chosen from an alphabet of finite cardinality, such that $\E[Q(x)] = x$. Our vector quantization algorithm (\cref{alg:quantizer}) achieves this, and is based on the following idea: Let $\cP$ be a convex polytope with vertices $\set{v_1,v_2,\ldots}$ such that $\set{x \in \R^d : \norm{x}_\infty \leq B} \subseteq \cP$. Given $x$ (with $\norm{x}_\infty \leq B$), express $x$ as a convex combination of vertices of $\cP$ (say, $x = \sum_{i} \theta_i v_i$) and output a random vertex $V$, where $V = v_i$ with probability $\theta_i$. Clearly, $\E[V] = x$.

Specifically, \cref{alg:quantizer} uses the polytope $\cP = \cP_{\cV}$ formed by the vertex set $\cV = \set{\pm (Bd)e_1,\ldots,\pm (Bd)e_d}$, where $e_i$ is the $i$-th standard basis vector (\ie $\cP$ is the $\lp[1]$ ball of radius $Bd$).
Note that $\abs{\cV} = 2d$ and that $\set{x \in \R^d : \norm{x}_\infty \leq B} \subseteq \cP_{\cV}$. 
This leads to the following claim.
\begin{claim} \label{clm:quantizer}
	Given $x \in R^d$ with $\norm{x}_\infty \leq B$ as input,~\cref{alg:quantizer} outputs a random variable $Q(x) \in \cV$ that is an unbiased estimate of $x$, with $\abs{\cV} = 2d$.
\end{claim}

\paragraph{Remark.} A more direct approach to quantization would be to do it coordinate-wise, \ie quantize (independently) each coordinate to $\set{-B,B}$ with appropriate probability to make it unbiased. This can equivalently be seen as quantizing the vector using the $\numbits_\infty$ ball  (of radius $B$) as the polytope. Here, the alphabet size becomes $2^d$ instead of $2d$ in \cref{alg:quantizer}; but, on the plus side, the coordinate-wise variance of the quantized vector becomes $\approx B^2$, instead of $\approx (Bd)^2$ in \cref{alg:quantizer}. In our estimators, we will be quantizing vectors of constant length ($d$), so these dependencies on $d$ do not affect the rate (up to constants).

\begin{algorithm}[H]
	Let $\cV = \set{\pm (Bd)e_1,\ldots,\pm (Bd)e_d}$. Label the vectors in $\cV$ as $v_1,\ldots,v_{2d}$.
	\begin{algorithmic}[1]
		\Require $x \in \R^d$ with $\norm{x}_\infty \leq B$. 				
		\State Write $x$ as convex combination of vectors in $\cV$: 
		$x = \sum_{i=1}^{2d} \theta_i v_i.$				
		\State Choose $I \in \set{1,\ldots,2d}$ randomly where $I=i$ with probability $\theta_i$ and \Return $Q(x) = v_I$.	
	\end{algorithmic}	
	\caption{Vector quantization} 
	\label{alg:quantizer}
\end{algorithm}

\paragraph{(iii) Distributed simulation.} The problem of \emph{distributed simulation} is the following: There are $n$ players, each having an i.i.d.\ sample from an unknown $d$-ary distribution $\bp$. Each player can only send $\numbits$ bits to a central referee, where $\numbits < \log d$. Can the referee simulate i.i.d.\ samples from $\bp$ using $\numbits$-bit messages from the players? \cite{AcharyaCT20inf} proposed a noninteractive communication protocol,
 denoted $\textsc{DistrSim}_\numbits$,
using which the referee can simulate {one} sample from $\bp$ using $\numbits$-bit messages from $O(d/2^\numbits)$ players. Moreover, the protocol is deterministic at the players, and only requires private randomness at the referee. 
\begin{theorem} \label{thm:simulation}
	For any $\numbits \geq 1$, the simulation protocol $\textsc{DistrSim}_\numbits$ lets the referee simulate $\Omega(n 2^\numbits/d)$ i.i.d.\ samples from an unknown $d$-ary probability distribution $\bp$ using $\numbits$-bit messages from $n$ players, where each player holds an independent sample from $\bp$.
\end{theorem}

\paragraph{Combining the ingredients.} We now discuss how the three ideas come together. To mimic the classical estimator~\eqref{eq:central:nonadapt}, a player with sample $X$ would ideally like to communicate $\set{\phi_{\high,k}(X)}_{k \in \Z}$, but cannot do so due to communication constraints. Wavelet-induced sparsity (\cref{clm:sparsity}) ensures that communicating the bin (out of $2^H$ possible bins) in which $X$ lies is tantamount to identifying the set of at most $d\eqdef 2(A+2)$ indices $k$ for which $\phi_{\high,k}(X)$ is possibly non-zero. Moreover, the player can quantize (unbiasedly) the vector containing values of $\phi_{\high,k}(X)$ at these indices using \cref{alg:quantizer}, whose output is one of $2d$ possibilities (\cref{clm:quantizer}). Thus, overall, using an alphabet of size at most $2^{\high} \cdot (2d) = O(2^\high)$, a player can communicate an unbiased estimate of $\set{\phi_{\high,k}(X)}_{k \in \Z}$. It can be shown that a density estimate based on these unbiased estimates from $n$ players still achieves centralized minimax rates (up to constants).
However, if $2^\numbits < 4(A+2)2^{\high}$, the players cannot send these estimates directly to the referee. In this case, the players and the referee use the distributed simulation protocol $\textsc{DistrSim}_\numbits$  (\cref{thm:simulation}), which, effectively, enables the referee to simulate $O(n2^\numbits/2^H)$ i.i.d.\ realizations of unbiased estimates of $\set{\phi_{\high,k}(X)}_{k \in \Z}$. The referee can now output a density estimate based on these simulated estimates. The degradation in minimax rates under communication constraints is due to the fact that the referee has only $O(n2^\numbits/2^H)$ realizations of unbiased estimates of $\set{\phi_{\high,k}(X)}_k$, instead of $n$. \medskip

We now give details of the idea outlined above. The resulting estimator (``single-level estimator'') is a communication-constrained version of the classical estimator given in~\eqref{eq:central:nonadapt}. We then describe an adaptive estimator (``multi-level estimator''), which is a communication-constrained version of the classical adaptive estimator~\eqref{eq:central:adapt}.

\subsection{Single-level estimator}
The $n$ players and the referee agree beforehand on the following: wavelet functions $\phi, \psi$; $\high \in \Z_+$; partition $[0,1] = \bigcup_{t=0}^{2^{\high}-1} E_t^{(\high)}$ as in~\eqref{eqn:bins}; collections of indices $\cA_t^{(\high)}$, $t \in \ibrac{0,2^{\high}-1}$ as in~\eqref{eqn:At}. For every $t$, the indices in $\cA_t^{(\high)}$ are arranged in ascending order.

\paragraph{Player's side (\cref{alg:singlevel:players}).}
Each player carries out two broad steps: (i)~quantization; and (ii)~simulation.

\begin{algorithm}[htp]
	\begin{algorithmic}[1]
		\For{$i = 1,\ldots,n$}
		\State Player $i$ computes $Z_i \eqdef (B_i,Q(V_i))$, where: (i) $B_i$ is the bin in which $X_i$ lies; (ii) $Q(V_i)$ is an unbiased quantization of the vector $V_i \eqdef \set{2^{-\high/2}\phi_{\high,k}(X_i)}_{k \in \cA_{B_i}^{(\high)}}$.  \Comment{\textbf{Quantization}}
		
		\State Player $i$ computes $\numbits$-bit message $Y_i$ corresponding to $Z_i$ as per $\textsc{DistrSim}_\numbits$ (\cref{thm:simulation}), and sends it to the referee.  \Comment{\textbf{Simulation}}
		\EndFor	
	\end{algorithmic}
	\caption{Single-level estimator (Players)}
	\label{alg:singlevel:players}
\end{algorithm}

The scaling by $2^{-\high/2}$ in the definition of $V_i$ (line 2 in \cref{alg:singlevel:players}) ensures that $\norm{V_i}_\infty \leq \norm{\phi}_\infty$, which is a constant. This enables the use of \cref{alg:quantizer} to compute quantization of $V_i$. Overall, computing $Z_i = (B_i,Q(V_i))$ involves two quantizations: $B_i$ can be seen as a quantized version of $X_i \in [0,1]$; $Q(V_i)$ is a quantized version of $\set{\phi_{\high, k}(X_i)}_k$. Moreover, for each $i \in \ibrac{n}$, $Z_i \in \cZ^{(\high)}$, where 
$
\cZ^{(\high)} \eqdef \ibrac{0,2^{\high}-1} \times \set{\pm Be_1,\ldots,\pm Be_d}
$
(with $d \leq 2(A+2)$, by \cref{clm:sparsity}), so that
$
\dabs{\cZ^{(\high)}} \leq 4(A+2) \: 2^{\high} = O(2^{\high}).
$ \medskip

Thus, $Z_1,\ldots,Z_n$ are i.i.d.\ samples (since $X_1,\ldots,X_n$ are i.i.d.) from a $\dabs{\cZ^{(\high)}}$-ary distribution (call it $\bp_{Z^{(\high)}}$) distributed across $n$ players, where \smash{$\dabs{\cZ^{(\high)}} = O(2^{\high})$}. Since a player can send only $\numbits$ bits, player $i$ cannot send $Z_i$ directly if $2^\numbits < \dabs{\cZ^{(\high)}}$. In this case, player $i$ computes an $\numbits$-bit message $Y_i$ according to the distributed simulation protocol $\textsc{DistrSim}_\numbits$, and sends $Y_i$ to the referee.

\paragraph{Referee's side (\cref{alg:singlevel:referee}).}
The referee, using the simulated i.i.d.\ samples from $\bp_{Z^{(\high)}}$, computes the density estimate similar to the classical estimate~\eqref{eq:central:nonadapt}. This is possible because the $m = O(n2^\numbits/2^H)$ simulated samples are i.i.d.\ realizations of unbiased quantization of $\set{\phi_{\high,k}(X)}_{k \in \Z}$.

\begin{algorithm}[htp]
	\begin{algorithmic}[1]
		\State From $Y_1,\ldots,Y_n$, referee obtains $m = O(n2^\numbits/\dabs{\cZ^{(\high)}}) = O(n2^\numbits/ 2^{\high})$ i.i.d.\ samples $Z'_1,\ldots,Z'_m \sim \bp_{Z^{(\high)}}$ as per $\textsc{DistrSim}_\numbits$, where $Z'_i = (B'_i, Q'_i) \in \cZ^{(\high)}$.
		
		\For{$i = 1,\ldots,m$}		
		\State Referee computes
		\begin{equation} \label{eqn:slcoeff}
			\widehat{\phi}_{\high, k}^{(i)} \eqdef 
			\begin{cases}
				2^{\high/2} \; Q'_i(k) & \text{if } k \in \cA_{B'_i}^{(\high)} \\
				0 & \text{otherwise}, 
			\end{cases}
		\end{equation}	
		where $Q'_i(k)$ is the entry in $Q'_i$ corresponding to index \smash{$k \in \cA_{B'_i}^{(\high)}$}. \Comment{Scaling by $2^{\high/2}$ is to negate the scaling by $2^{-\high/2}$ used in definition of $V_i$ on the players' side.}
		\EndFor
		
		\State Referee outputs density estimate
		\begin{equation} \label{eqn:sldensity}
			\widehat{f} = \sum_{k \in \Z} \widehat{\alpha}_{\high,k} \phi_{\high,k},
			\quad \text{where } \ \widehat{\alpha}_{\high,k} = \frac{1}{m} \sum_{i=1}^{m} \widehat{\phi}_{\high,k}^{(i)}, \ k \in \Z.
		\end{equation}
	\end{algorithmic}
	\caption{Single-level estimator (Referee)}
	\label{alg:singlevel:referee}
\end{algorithm}
 
\paragraph{Result.} For $H$ such that \smash{$2^H \asymp \min\{ (n2^\numbits)^{\frac{1}{2s+2}}, n^{\frac{1}{2s+1}} \}$}, the single-level estimator recovers the guarantees in \cref{theo:ub:nonadaptive} (see~\cref{supp:singlevel} for details). The estimator is nonadaptive because setting $H$ requires knowing Besov parameter $s$. Further, note that the estimator is indeed noninteractive, as player $i$'s message $Y_i$ does not depend on messages $Y_1,\ldots,Y_{i-1}$.

\subsection{Multi-level estimator: An adaptive density estimator}
The key observation in designing our multi-level estimator is that different coefficients need to be recovered with different accuracy. We enable this by dividing players into groups for estimating different coefficients, and using a different level of quantization for each group. This is in contrast to simply mimicking the classical adaptive estimator~\eqref{eq:central:adapt}, which would suggest that a player with sample $X$ should quantize and communicate information about $\set{\phi_{\low, k}(X)}_k, \set{\set{\psi_{J,k}(X)}_k}_{J \in \ibrac{\low,\high}}$. Instead, we do the following: 
Divide $n$ players into $\high - \low + 1$ groups of equal size (so,
each group has $n' = \frac{n}{\high-\low+1}$ players). Label the
groups $\low,\low+1,\ldots,\high$. Players in group $L$ only focus on
$\set{\phi_{\low, k}(X)}_k, \set{\psi_{\low, k}(X)}_k$. Players in
group $J$, $J \in \ibrac{\low+1,\high}$, only focus on $\set{\psi_{J,
k}(X)}_k$. Moreover, players in group $J$, $J \in \ibrac{\low,\high}$,
quantize their sample $X$ using $2^J$ uniform bins. As before, by
\cref{clm:sparsity}, this is tantamount to identifying at most a
constant number of indices for which the wavelet function evaluates to
a non-zero value (since players in group $J$ only consider
$\phi_{J,k}$ or $\psi_{J, k}$). The player then quantizes the vector
containing these values using \cref{alg:quantizer}, before
using distributed simulation.  \medskip

The $\ns$ players and the referee agree beforehand on the following: wavelet functions $\phi, \psi$; $\low,\high \in \Z_+$; division of players into $\high - \low + 1$ groups. Further, for each $J \in \ibrac{\low,\high}$, the $n'$ players in group $J$ and the referee agree on the following: partition $[0,1] = \bigcup_{t=0}^{2^{J}-1} E_t^{(J)}$ as in~\eqref{eqn:bins}; collection of indices $\cA_t^{(J)}, \cB_t^{(J)}$, $t \in \ibrac{0,2^{J}-1}$, as in~\eqref{eqn:At},~\eqref{eqn:Bt}. For every $J,t$, the indices in $\cA_t^{(J)}, \cB_t^{(J)}$ are arranged in ascending order.

\paragraph{Player's side (\cref{alg:multlevel:players}).}
Label players in group $J$ as $(1,J),\ldots,(n',J)$. We denote by $X_{i,J}$ the sample with player $(i,J)$. Essentially, players in group $J$ run quantization and simulation steps as in the single-level algorithm (\cref{alg:singlevel:players}), with $H$ replaced by $J$.

\begin{algorithm}[htp]
	\begin{algorithmic}[1]
		\For{$J = \low,\low+1,\ldots,\high$}
		\For{$i = 1,\ldots,n'$}
		\State Player $(i,J)$ computes $Z_{i,J} = (B_{i,J},Q(V_{i,J}))$, where: (i) $B_{i,J}$ is the bin (out of $2^J$ bins) in which $X_{i,J}$ lies; (ii) $Q(V_{i,J})$ is an unbiased quantization of the vector $V_{i,J}$, where
		\[V_{i,J} \eqdef 
		\begin{cases}
			\set{2^{-J/2}\psi_{J,k}(X_{i,J})}_{k \in \cB_{B_{i,J}}^{(J)}}
				&\text{if } J \in \ibrac{\low+1,\high}, \\
			{\tt CONCAT}\big({ \set{2^{-\low/2}\phi_{\low,k}(X_{i,\low})}_{k \in \cA_{B_{i,\low}}^{(\low)}},
			\set{2^{-\low/2}\psi_{\low,k}(X_{i,\low})}_{k \in \cB_{B_{i,\low}}^{(\low)}} } \big)
				&\text{if } J = \low.
		\end{cases}
		\]	
		(${\tt CONCAT}$ denotes concatenation of two vectors.)  \Comment{\textbf{Quantization}}

		\State Player $(i,J)$ computes $\numbits$-bit message $Y_{i,J}$ corresponding to $Z_{i,J}$ as per $\textsc{DistrSim}_\numbits$, and sends it to the referee.  \Comment{\textbf{Simulation}}
		\EndFor
		\EndFor	
	\end{algorithmic}
	\caption{Multi-level estimator (Players)}
	\label{alg:multlevel:players}
\end{algorithm}

\paragraph{Referee's side (\cref{alg:multlevel:referee}).} 
For players in group $J$, $Z_{i,J} \in \cZ^{(J)}$, where $\abs{\cZ^{(J)}} = O(2^{J})$. Thus, after distributed simulation, the referee obtains $m_J = O(n' 2^\numbits/2^J) = O(n 2^\numbits/(H-L+1)2^J)$ samples from players of group $J$. Note that, higher the $J$, fewer the simulated samples; this dependence on $J$ of the number of samples available with the referee is one of the major differences between the classical and the distributed setting.
Finally, using the simulated samples from players of every group, referee computes a density estimate similar to the adaptive classical estimator~\eqref{eq:central:adapt}, with threshold value $t_J = \kappa \sqrt{J/m_J}$, for a constant $\kappa$.

\begin{algorithm}[htp] 
	\begin{algorithmic}[1]
		\For{$J = \low,\low+1,\ldots,\high$}
		\State From $Y_{1,J},\ldots,Y_{n',J}$, referee obtains $m_J = O(n2^\numbits/(H-L+1) 2^{J})$ i.i.d.\ samples 
		$Z'_{1,J},\ldots,Z'_{m_J,J} \sim \bp_{Z^{(J)}}$ as per $\textsc{DistrSim}_\numbits$, where $Z'_{i,J} = (B'_{i,J}, Q'_{i,J}) \in \cZ^{(J)}$.
		
		\For{$i = 1,\ldots,m_J$}
		\If{$J = \low$}
		\State Referee computes $\set{\widehat{\phi}_{\low,k}^{(i)}}_{k \in \Z}$ as
		$
			\widehat{\phi}_{\low, k}^{(i)} \eqdef 
			\begin{cases}
				2^{\low/2} \; Q'_{i,\low}(k) & \text{if } k \in \cA_{B'_{i,\low}}^{(\low)}\\
				0 & \text{ otherwise.}
			\end{cases}
		$
		\EndIf		
		
		\State Referee computes $\set{\widehat{\psi}_{J,k}^{(i)}}_{k \in \Z}$ as
		$
			\widehat{\psi}_{J, k}^{(i)} \eqdef 
			\begin{cases}
				2^{J/2} \; Q'_{i,J}(k) & \text{if } k \in \cB_{B'_{i,J}}^{(J)}\\
				0 & \text{ otherwise.}
			\end{cases}
		$
		\EndFor
		\EndFor
		
		\State Referee outputs density estimate 
		\begin{equation} \label{eqn:mldensity}
			\widehat{f} = \sum_{k} \widehat{\alpha}_{\low,k} \phi_{\low,k} + \sum_{J=\low}^{\high} \sum_{k} \tilde{\beta}_{J,k} \psi_{J,k},
		\end{equation}
	where $\widehat{\alpha}_{\low,k} = \frac{1}{m_\low} \sum_{i=1}^{m_\low} \widehat{\phi}_{\low,k}^{(i)}$, $\ \widehat{\beta}_{J,k} = \frac{1}{m_J} \sum_{i=1}^{m_J} \widehat{\psi}_{J,k}^{(i)}$, $\ \tilde{\beta}_{J,k} = \widehat{\beta}_{J,k} \ind{|\widehat{\beta}_{J,k}| \geq t_J \eqdef \kappa\sqrt{J/m_J}}$.
	\end{algorithmic}	
	\caption{Multi-level algorithm (Referee)}
	\label{alg:multlevel:referee}
\end{algorithm}

\paragraph{Result.} For $L,H$ satisfying $2^L \asymp  \min\{(n2^\numbits)^{\frac{1}{2(N+1)+2}},n^{\frac{1}{2(N+1)+1}}\}$ and $2^H \asymp  \min\{\sqrt{(n2^\numbits)}/\log (n2^\numbits),n/\log n\}$, the multi-level estimator yields the guarantees in \cref{theo:ub:adaptive} (see~\cref{supp:multlevel} for details) as long as $s \in (1/p,N+1)$ (recall that $N$ is the regularity of the wavelet basis). Since $L,H$ do not depend on specific Besov parameters, the estimator is adaptive. Moreover, it is noninteractive.

	\section{Lower Bounds} \label{sec:lowerbound}
		We conclude with a description of our information-theoretic lower bounds (\cref{theo:lb}) for the minimax loss $\cL_r^*(\ns, \numbits, p,q,s)$, which applies to the broader class of interactive protocols (recall that our matching upper bounds are obtained by noninteractive ones); the details can be found in~\cref{supp:lowerbound}. To derive lower bounds, we consider a family of probability distributions $\cP$ parameterized by $\set{-1,1}^d$ for some $d \in \Z_+$; that is, $\cP = \{\p_z : z \in \set{-1,1}^d\}$, where $\p_z$ has density $f_z$. Moreover, we specify a prior $\pi$ on $Z = (Z_1,\ldots,Z_d) \in \set{-1,1}^d$, defined as $Z_i \sim {\tt Rademacher}(\tau)$ independently for each $i \in [d]$, for some $\tau \in (0,1/2]$. We then consider the following scenario: \medskip

For $Z \sim \pi$, let $X_1,\ldots,X_{\ns}$ be i.i.d.\ samples from $\p_Z$ distributed across $\ns$ players. Let $Y_1,\ldots,Y_{\ns}$ be $\numbits$-bit messages sent by the players (possibly interactively) to the referee. Denote by $\p^{Y^n}_{+i}$ (resp. $\p^{Y^n}_{-i}$) the joint distribution of $Y_1,\ldots,Y_{\ns}$, given $Z_i = 1$ (resp. $Z_i = -1$). That is,
\begin{equation}
  \label{eq:def:p+1:p-1}
\p^{Y^n}_{+i} = \frac{1}{\tau} \sum_{z:z_i=1}\pi(z)\p^{Y^n}_z, \quad
\p^{Y^n}_{-i} = \frac{1}{1-\tau} \sum_{z:z_i=-1}\pi(z)\p^{Y^n}_z,
\end{equation}
where $\p^{Y^n}_z$ is the joint distribution of $Y_1,\ldots,Y_n$, given $Z=z$. \medskip

In this scenario, we will analyze the ``average discrepancy''
$\frac{1}{d}\sum_{i=1}^{d}\totalvardist{\p^{Y^n}_{-i}}{\p^{Y^n}_{+i}}$, where $\totalvardist{\p}{\q}$ denotes the total variation distance between $\p$ and $\q$. On the one hand, a result from \cite{AcharyaCT20} will give us an \emph{upper bound} on this average discrepancy as a function of $\ns$ and $\numbits$ which holds for \emph{any} interactive protocol generating $Y_1,\ldots,Y_{\ns}$ (\cref{thm:AcharyaCT20}). On the other hand, we will derive a \emph{lower bound} on average discrepancy (as a function of the error rate $\dst$) as follows: Consider a communication-constrained density estimation algorithm (possibly interactive) which outputs $\hat{f}$ satisfying $\sup_{f \in \cB(p,q,s)} \bE{f}{\| \hat f-f \|_r^r} \leq \dst^r$. We will show that one can use the messages $Y_1,\ldots,Y_{\ns}$ generated by this algorithm to solve, for each $i \in \ibrac{d}$, the binary hypothesis testing problem of deciding whether $Z_i = 1$ or $Z_i=-1$. This, in turn, will imply a lower bound on $\frac{1}{d}\sum_{i=1}^{d}\totalvardist{\p^{Y^n}_{-i}}{\p^{Y^n}_{+i}}$. Putting together the upper and lower bounds on average discrepancy will give us a lower bound on $\dst$. \medskip

The parameterized family of distributions $\cP$ is constructed as follows: Let $f_0$ be a function supported on $[0,1]$. Let $I_1,\ldots,I_d \subseteq [0,1]$ be mutually disjoint intervals of equal length. Let $\psi_i$ be a ``bump'' function supported on interval $I_i$, where $\psi_i$'s are all translations of the same bump function. Then, for $z = (z_1,\ldots,z_d) \in \set{-1,1}^d$, we define $p_z$ to be a probability distribution with density $f_z$, defined as the ``baseline'' $f_0$ perturbed by adding (a rescaling of) the bump $\psi_i$ according to the value of $z_i$. In more detail, to get the desired lower bounds, we distinguish two cases depending on whether $r<(s+1)p$, and construct two families of distributions: $\cP_1$ (when $r < (s+1)p$) and $\cP_2$ (when $r \geq (s+1)p$).

\begin{figure}
    \centering
    \begin{minipage}{0.49\textwidth}
        \centering
        \includegraphics[width=1.0\textwidth]{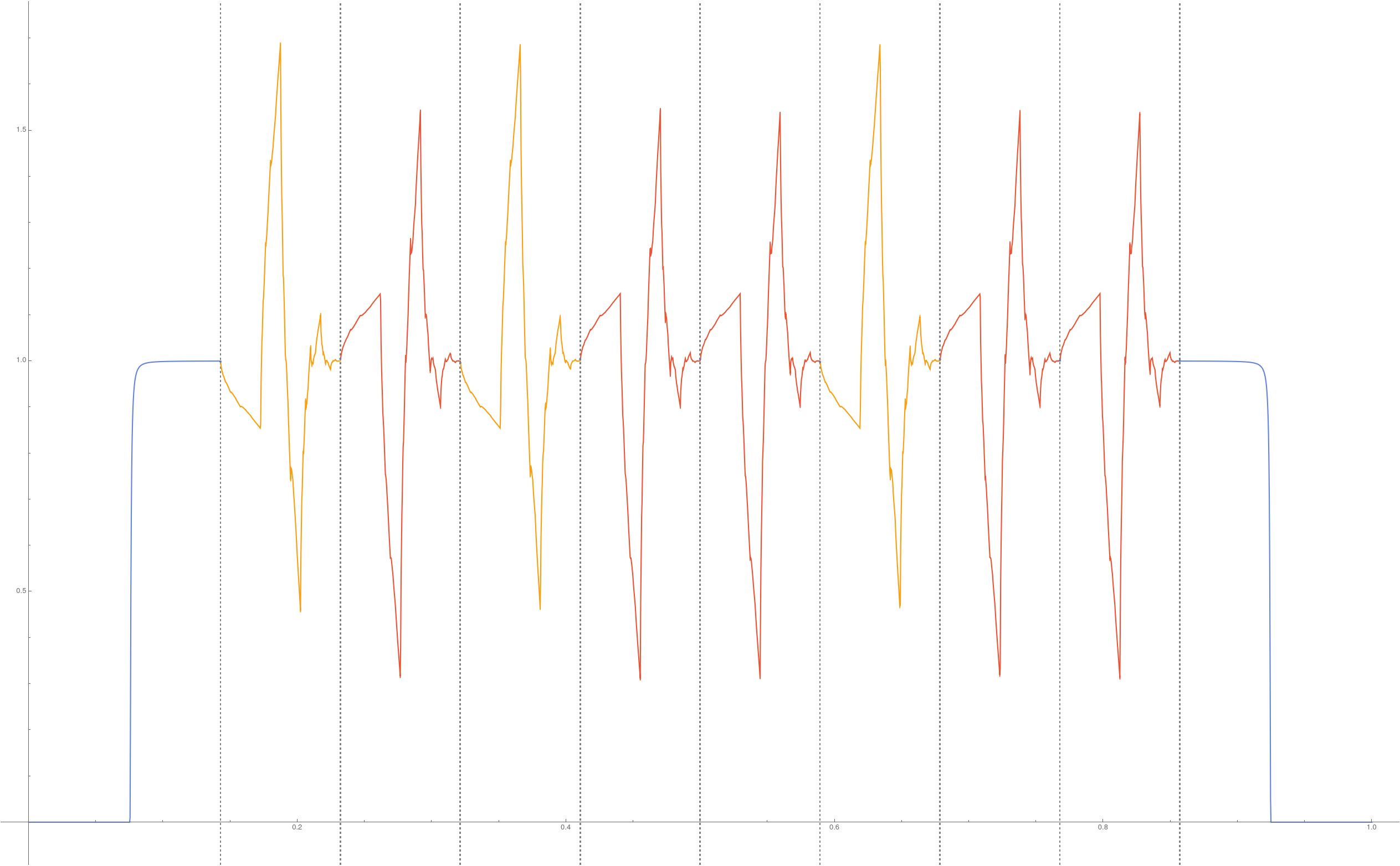}
        \caption{\label{fig:lb1}\small{}An illustration of the construction for the first lower bound, $\cP_1$, in the regime $r < (s+1)p$. Each of the $d=8$ bumps depicted is flipped depending of the value of $z_i$ for this particular instance of the construction (here, $z=(+1,-1,+1,-1,-1,+1,-1,-1)$).}
    \end{minipage}\hfill
    \begin{minipage}{0.49\textwidth}
        \centering
        \includegraphics[width=1.0\textwidth]{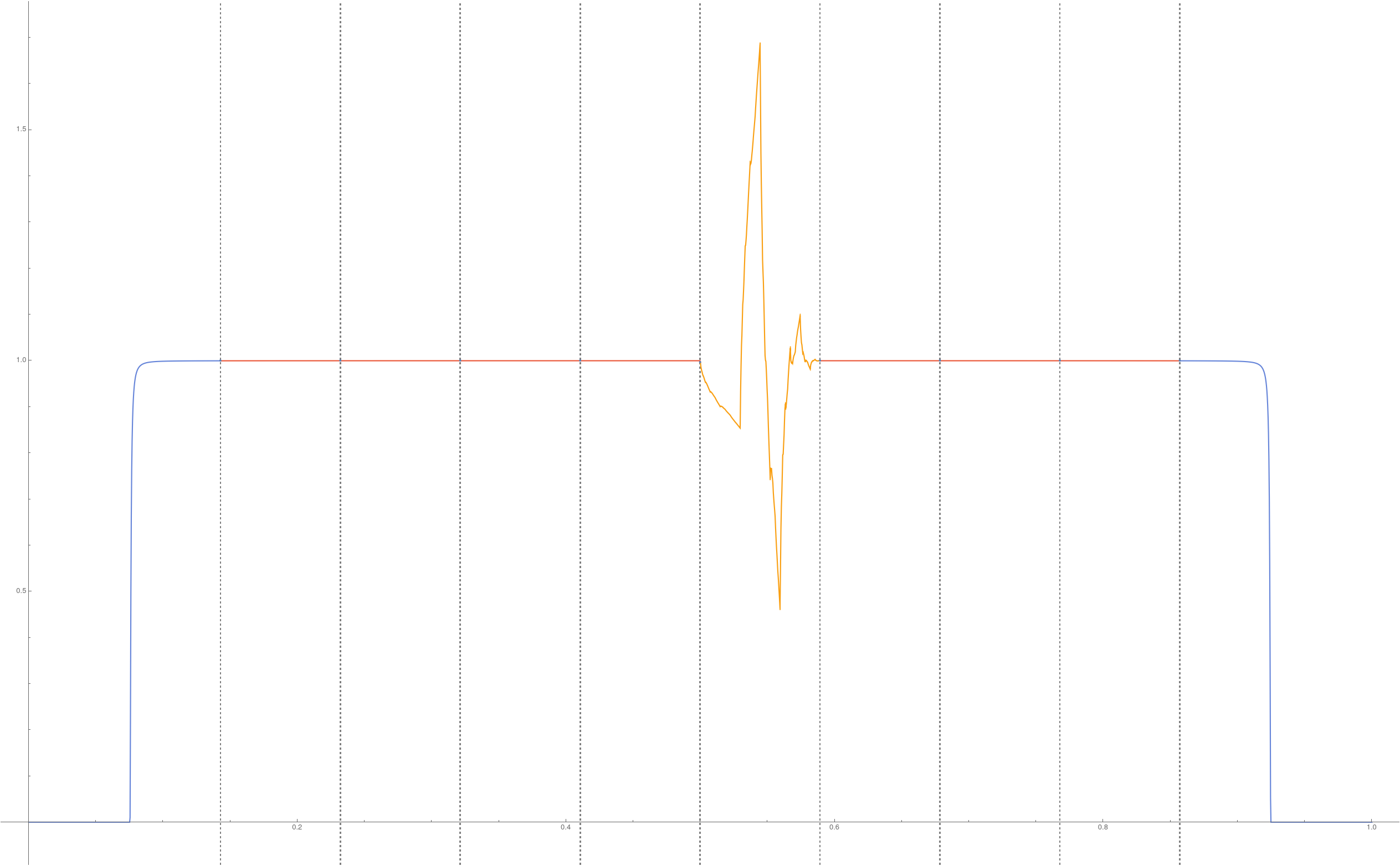}
        \caption{\label{fig:lb2}\small{}An illustration of the construction for the second lower bound, $\cP_2$, in the regime $r \geq (s+1)p$. Each of the $d=8$ bumps depicted is either present or not, depending of the value of $z_i$ for this particular instance of the construction (here, $z=(-1,-1,-1,-1,+1,-1,-1,-1)$).}
    \end{minipage}
\end{figure}

\begin{itemize}
\item For $\cP_1$, we use a uniform prior on $Z = (Z_1,\ldots,Z_d)$, \ie $Z$ has independent Rademacher coordinates, and set
\[
  f_z = f_0 + \gamma\sum_{i=1}^{d}z_i \psi_i
\] for some suitably small parameter $\gamma>0$. That is, the baseline density $f_0$ has disjoint bumps, which are either $\psi_i$ or $-\psi_i$ depending on the value of $z_i$. See~\cref{fig:lb1} for an illustration, and~\cref{sec:lb1} for the details.
\item For $\cP_2$, we use a non-uniform (``sparse'') prior on $Z$, where $Z_1,\dots,Z_d$ are independent with parameter $1/d$, and we set 
\[
  f_z = f_0 + \gamma\sum_{i=1}^{d}(1+z_i) \psi_i
\] (so that bump $\psi_i$ only appears if $Z_i=1$). See~\cref{fig:lb1} for an illustration, and~\cref{sec:lb1} for details.
\end{itemize}
Applying to these constructions the method described above allows us to derive the lower bounds of \cref{theo:lb}.

\begin{remark}
We note that a similar proof would enable us to derive an analogous result for local privacy, thus extending the lower bounds of~\cite{ButuceaDKS20} (which are restricted to noninteractive protocols) to the interactive setting.
\end{remark}

	\printbibliography
	\appendix
	\section{Useful facts} \label{supp:facts}
		We recall some results that will be used in our analysis.

\subsection{Useful facts about Besov spaces}
We record a few facts about Besov spaces that will be used in our analysis. Throughout, we assume that $f \in \cB(p,q,s)$ with $\norm{f}_{spq} \leq 1$ and ${\rm supp}(f) \subseteq [0,1]$. The following fact is apparent from our discussion on wavelets.
\begin{fact} \label{fact:waveletlevel}
	Let the wavelet expansion of $f$ be as in \eqref{eqn:waveletexp}. For $\high \geq \low$, define $f^{(\high)} := \sum_{k \in \Z} \alpha_{\low,k} \phi_{\low,k} + \sum_{j = \low}^{\high-1} \sum_{k \in \Z} \beta_{j,k} \psi_{j,k}$. Then, $f^{(\high)} = \sum_{k \in \Z} \alpha_{\high,k} \phi_{\high,k}$.
\end{fact}
Since ${\rm supp}(f) \subseteq [0,1]$, and ${\rm supp}(\phi_{j,0}), {\rm supp}(\psi_{j,0}) \subseteq [-A2^{-j},A2^{-j}]$, there is no overlap between ${\rm supp}(f)$ and ${\rm supp}(\phi_{j,k}), {\rm supp}(\psi_{j,k})$ for all but a finite number of indices $k$. In particular, we have the following.
\begin{fact} \label{fact:coeffnonzero}
	Let the wavelet expansion of $f$ be as in \eqref{eqn:waveletexp}. Then, for any given $j \in \Z_+$, there are $O(2^j)$ translation indices $k$ such that $\phi_{j,k}(x)$ or $\psi_{j,k}(x)$ is possibly non-zero, where $x \in {\rm supp}(f)$.
\end{fact}
For $f \in \cB(p,q,s)$, it is clear from the definition of Besov norm that the wavelet coefficients must decay sufficiently fast. More precisely, we have the following.
\begin{fact}
	\label{fact:useful:bound:beta:norm}
	If $f\in\besov(p,q,s)$, then
	\[
	\lim_{j\to\infty}2^{jp(s+\frac{1}{2}-\frac{1}{p})}\sum_{k\in \Z} |\beta_{j,k}|^p = 0
	\]
	and in particular there exists $C>0$ such that
	\[
	\norm{\beta_j}_p^p = \sum_{k\in \Z} |\beta_{j,k}|^p \leq C\cdot 2^{-jp(s+\frac{1}{2}-\frac{1}{p})}.
	\]
\end{fact}
The next fact quantifies the approximation error when the wavelet expansion of $f$ is truncated.
\begin{fact} \label{fact:centralbias}
	Let $f^{(\high)}$ be as in Claim \ref{fact:waveletlevel}. Then, for $r \geq 1$,
	\[
	\norm{f^{(\high)} - f}_r \leq C
	\begin{cases}
		2^{-\high s} &\text{if } r \leq p, \\
		2^{-\high (s-1/p+1/r)} &\text{if } r > p.
	\end{cases}
	\]
\end{fact}
The next fact (from equation (15) in \cite{DonohoJKP96}) gives a bound on $\norm{f}_\infty$ when $\norm{f}_{spq} \leq 1$.
\begin{fact} \label{fact:infinitynorm}
	Let $s > 1/p$. Then
	\[
	\norm{f}_\infty \leq \paren{1 - 2^{-(s-1/p)q'}}^{1/q'}
	\]
	where $1/q + 1/q' = 1$.
\end{fact}

Now, let $X_1, \ldots, X_n$ be independent samples from distribution with density $f$. For $j,k \in \Z$, define
\begin{equation} \label{eqn:centralcoeff0}
	\bar{\alpha}_{j,k} \eqdef \frac{1}{n} \sum_{i=1}^{n}\phi_{j,k}(X_i), \quad
	\bar{\beta}_{j,k} \eqdef \frac{1}{n} \sum_{i=1}^{n}\psi_{j,k}(X_i).
\end{equation}
Observe that $\bar{\alpha}_{j,k}$ (resp., $\bar{\beta}_{j,k}$) is an unbiased estimate of $\alpha_{j,k}$ (resp., $\beta_{j,k}$). The following fact is from equation (16) in \cite{DonohoJKP96}.
\begin{fact} \label{fact:centralcoeff:rthnorm}
	Let $n \geq 2^j$. Then, for $r \geq 1$,
	\begin{align*}
		\bEE{\abs{\bar{\alpha}_{j,k} - {\alpha}_{j,k}}^r} \leq C n^{-r/2}, \quad 
		\bEE{\abs{\bar{\beta}_{j,k} - {\beta}_{j,k}}^r} \leq C n^{-r/2}
	\end{align*}
	where $C$ is a constant that depends on $p,q,s,r,\phi,\psi$.
\end{fact} 
We note another useful fact (obtained after setting $\beta = 0$ in equation (21) in \cite{DonohoJKP96}).
\begin{fact} \label{fact:useful:rthnorm}
	Let $g = \sum_{j=\low}^{\high} \sum_{k\in\Z} \widehat{g}_{j,k} \psi_{j,k}$, where $\widehat{g}_{j,k}$ is random. Then, for $r\geq 1$,
	\[
	\bEE{\norm{g}_r^r} \leq C (\high-\low)^{(r/2-1)_{+}} \sum_{j=\low}^{\high} 2^{j(r/2-1)} \sum_{k\in \Z} \bEE{ |\widehat{g}_{j,k}|^r }
	\]
	where $C$ is a constant that depends on $r$.
\end{fact}
\subsection{Useful probabilistic inequalities}

\begin{theorem}[Rosenthal's inequality \cite{Rosenthal70}] \label{theo:rosenthal}
	Let $X_1,\ldots,X_n$ be independent random variables such that $\bEE{X_i} = 0$. and $\bEE{\abs{X_i}^r} < \infty$ for every $i$. 
	\begin{enumerate}
		\item Suppose $\bEE{X_i^2} < \infty$ for every $i$. Then, for $1 \leq r \leq 2$,
		\[
		\bEE{\abs{\sum_{i=1}^{n}X_i}^r} \leq  \paren{\sum_{i=1}^{n}\bEE{X_i^2}}^{r/2}.
		\]	
		(This just follows from concavity of $f(x) = x^{r/2}$ for $r \leq 2$.)
		\item Suppose $\bEE{\abs{X_i}^r} < \infty$ for every $i$. Then, for $r > 2$, there exists a constant $K_r$ depending only on $r$ such that
		\[
		\bEE{\abs{\sum_{i=1}^{n}X_i}^r} \leq K_r \left\{ \sum_{i=1}^{n}\bEE{\abs{X_i}^r} + \paren{\sum_{i=1}^{n}\bEE{X_i^2}}^{r/2} \right\}.
		\]
	\end{enumerate}	
\end{theorem}

\begin{theorem}[Bernstein's inequality] \label{theo:bernstein}
	Let $X_1,\dots, X_n$ be independent random variables such that $\abs{X_i}\leq b$ almost surely, and $\bEE{X_i^2} \leq v_i$ for every $i$. Let $X \eqdef \sum_{i=1}^n X_i$ and $V \eqdef \sum_{i=1}^n v_i$. Then, for every $u\geq 0$,
	\[
	\bPr{ \abs{X-\bEE{X}} \geq u } \leq {\rm exp}\paren{-\frac{u^2}{2(V+\frac{bu}{3})}}
	\]
\end{theorem}
	\section{Analysis of single-level estimator} \label{supp:singlevel}
		Our goal is to upperbound the worst-case $\cL^r$ loss $\bEE{\norm{\widehat{f}-f}_r^r}$, where $\widehat{f}$ is the estimate output by the referee. Arguments and results in Sections~\ref{sec:coupling} and~\ref{sec:keylemmas} will be used in the analysis of the multi-level estimator as well.

\subsection{Coupling of simulated and ideal estimators} \label{sec:coupling}
Denote by $\bp$ the probability distribution corresponding to density $f$. Recall that $\bp_{\cZ^{(\high)}}$ is the distribution after quantization of samples from $\bp$. Suppose the referee has $m$ samples from $\bp_{\cZ^{(\high)}}$ using which it outputs the estimate $\widehat{f}$. 
To compute $\bEE{\norm{\widehat{f}-f}_r^r}$, consider the following statistically equivalent situation:
\begin{itemize}
	\item There are $m$ i.i.d. samples $X_1,\ldots,X_m \sim \bp$.
	
	\item For each $i \in \ibrac{m}$, let
	\begin{equation} \label{eqn:slcoeff1}
		\widehat{\phi}_{\high,k}(X_i) = 
		\begin{cases}
			0 & \text{if } k \notin \cA_{B_i}^{(\high)},  \\
			2^{\high/2} Q(V_i)(k) & \text{if } k \in \cA_{B_i}^{(\high)},
		\end{cases}
	\end{equation}
	where $B_i$ is the bin $X_i$ lies in, $Q(V_i)$ is obtained by quantizing $\set{2^{-\high/2}\phi_{\high,k}(X_i)}_{k \in \cA_{B_i}^{(\high)}}$ using Algorithm 1, and $Q(V_i)(k)$ is the entry in $Q(V_i)$ corresponding to $k \in \cA_{B_i}^{(\high)}$. 
	In other words, $\set{\widehat{\phi}_{\high,k}(X_i)}_{k \in \Z}$ is the quantized version of $\set{\phi_{\high,k}(X_i)}_{k \in \Z}$.
	
	\item Define $\widehat{f}$ as
	\begin{equation} \label{eq:slest}
	\widehat{f} = \sum_{k \in \Z} \widehat{\alpha}_{\high,k} \phi_{\high,k},
	\quad \text{where } \ \widehat{\alpha}_{\high,k} = \frac{1}{m} \sum_{i=1}^{m} \widehat{\phi}_{\high,k}(X_i), \ k \in \Z.
	\end{equation}
\end{itemize}
Then, computing the $\cL^r$ loss for the single-level estimator is equivalent to computing the $\cL^r$ loss for $\widehat{f}$ defined in \eqref{eq:slest}. From here on, when we refer to $\widehat{f}$, we mean $\widehat{f}$ defined in \eqref{eq:slest}. Now,
\begin{equation} \label{eq:lossbreakup}
\bEE{\norm{\widehat{f}-f}_r^r} \leq 2^{r-1}\paren{ \bEE{\norm{\widehat{f}-\bar{f}}_r^r} + \bEE{\norm{\bar{f}-f}_r^r }}
\end{equation}
where $\bar{f}$ is defined as
\begin{equation} \label{eq:slclassic}
\bar{f} = \sum_{k \in \Z} \bar{\alpha}_{\high,k} \phi_{\high,k} \quad \text{where} \quad \bar{\alpha}_{\high,k} = \frac{1}{m} \sum_{i=1}^{m} \phi_{\high,k}(X_i), \ k \in \Z.
\end{equation}
Note that this is just the classical density estimate obtained using $m$ samples $X_1,\ldots,X_m$. The idea behind introducing this coupling is to facilitate analysis by bringing in the classical density estimate \eqref{eq:slclassic} and breaking up the $\cL^r$ loss as in \eqref{eq:lossbreakup}.

\subsection{Key lemmas} \label{sec:keylemmas}
Since these lemmas will be used in the analysis of the multi-level estimator as well, we discuss them in more generality than would be required for only analyzing the single-level estimator. For $j,k \in \Z$, define
\begin{align}
	&\widehat{\alpha}_{j,k} \eqdef \frac{1}{m} \sum_{i=1}^{m} \widehat{\phi}_{j,k}(X_i), &&\widehat{\beta}_{j,k} \eqdef \frac{1}{m} \sum_{i=1}^{m} \widehat{\psi}_{j,k}(X_i), \label{eqn:quantizedcoeff} \\
	&\bar{\alpha}_{j,k} \eqdef \frac{1}{m} \sum_{i=1}^{m}\phi_{j,k}(X_i),
	&&\bar{\beta}_{j,k} \eqdef \frac{1}{m} \sum_{i=1}^{m}\psi_{j,k}(X_i), \label{eqn:centralcoeff}	
\end{align}
where $\set{\widehat{\phi}_{j,k}(X_i)}_{k \in \Z}$ and $\set{\widehat{\psi}_{j,k}(X_i)}_{k \in \Z}$ are quantized versions of $\set{\phi_{j,k}(X_i)}_{k \in \Z}$ and $\set{\psi_{j,k}(X_i)}_{k \in \Z}$, respectively (as in \eqref{eqn:slcoeff1}, with $\high$ replaced by $j$; $\phi$ replaced by $\psi$, and $\cA_{B_i}^{(j)}$ replaced by $\cB_{B_i}^{(j)}$  in the case of $\widehat{\beta}_{j,k}$). 
The following claim bounds the error between quantized and unquantized (classical) estimates of wavelet coefficients.
\begin{claim}[Error between quantized and unquantized estimates]
	\label{clm:coeff:central}
	For $r \geq 1$, we have
	\begin{align*}
		\bEE{\abs{\widehat{\alpha}_{j,k} - \bar{\alpha}_{j,k}}^r} \leq C
		\begin{cases}
			\frac{1}{m^{r/2}}, &\text{if } r \in [1,2],
                        \\
			\frac{2^{j(r/2-1)}}{m^{r-1}} + \frac{1}{m^{r/2}}, &\text{if } r >2,
		\end{cases}
	\end{align*}
	for a constant $C$. The same bound holds for $\bEE{\abs{\widehat{\beta}_{j,k} - \bar{\beta}_{j,k}}^r}$ as well.
\end{claim}
\begin{proof}
	For a given $j,k$,
	\[
	\bEE{\abs{\widehat{\alpha}_{j,k} - \bar{\alpha}_{j,k}}^r}
	= \bEE{ \abs{ \frac{1}{m} \sum_{i=1}^{m} \Paren{\widehat{\phi}_{j,k}(X_i) - \phi_{j,k}(X_i)} \indic{\cA_{B_i}^{(j)} \ni k} }^r } 
	= \frac{1}{m^r} \bEE{ \abs{\sum_{i=1}^{m} Y_{ik}}^r }
	\]
	where
	\[
	Y_{ik} \eqdef \Paren{\widehat{\phi}_{j,k}(X_i) - \phi_{j,k}(X_i)} \indic{ k \in \cA_{B_i}^{(j)} }.
	\]
	Note that, since the quantization is unbiased, we have $\bEE{ Y_{ik} } = 0$. Moreover, $\abs{Y_{ik}} \lesssim 2^{j/2}$ almost surely. 
	We first consider the case $r > 2$. Then, by Rosenthal's inequality (Theorem \ref{theo:rosenthal}), 
	\begin{equation} \label{eqn:rosen}
		\bEE{ \abs{\sum_{i=1}^{m} Y_{ik}}^r }
		\lesssim (2^{j/2})^{(r-2)} \sum_{i=1}^{m} \bEE{{Y_{ik}}^2}  + \Paren{\sum_{i=1}^{m} \bEE{ Y_{ik}^2}}^{\frac{r}{2}} 
		= 2^{j(\frac{r}{2}-1)}  m \ \bEE{ {Y_{1k}}^2 } + m^{\frac{r}{2}} \bEE{Y_{1k}^2}^{\frac{r}{2}}
	\end{equation}
	Moreover,
	\[
	\bEE{Y_{ik}^2}
	= \bEE{ \Paren{\widehat{\phi}_{j,k}(X_i) - \phi_{j,k}(X_i)}^2 \indic{ k \in \cA_{B_i}^{(j)} } } \\
	\lesssim 2^j \bPr{ k \in \cA_{B_i}^{(j)} }.
	\]
	Now, note that
	\[
	\bPr{ k \in \cA_{B_i}^{(j)} } = \bPr{ X_i \in {\rm supp}(\phi_{j,k})} \leq \frac{2A}{2^j} \norm{f}_\infty \lesssim \frac{1}{2^j} \ \paren{\text{using Fact \ref{fact:infinitynorm}}}
	\]
	which gives
	\[
	\bEE{Y_{ik}^2} \lesssim 1.
	\]
	Substituting this in \eqref{eqn:rosen}, we get the desired result when $r > 2$. For $r \in [1,2]$, using part (1) of Theorem \ref{theo:rosenthal}, only the second term in \eqref{eqn:rosen} remains. This gives the result for $r \in [1,2]$. The proof for $\bEE{\abs{\widehat{\beta}_{j,k} - \bar{\beta}_{j,k}}^r}$ is analogous.
\end{proof}

The claim above, combined with Fact \ref{fact:centralcoeff:rthnorm}, lets us bound the error between quantized estimates and true coefficients as follows.
\begin{claim}[Error between quantized estimates and true coefficients]
	\label{clm:coeff}
	Let $m \geq 2^j$. Then, for $r \geq 1$, we have
	\[
	\bEE{\abs{\widehat{\alpha}_{j,k} - {\alpha}_{j,k}}^r} \leq \frac{C}{m^{r/2}}
	\,,\qquad
	\bEE{\abs{\widehat{\beta}_{j,k} - {\beta}_{j,k}}^r} \leq \frac{C}{m^{r/2}}
	\]
	for a constant $C$.
\end{claim}
\begin{proof}
	Note that
	\[
	\bEE{\abs{\widehat{\alpha}_{j,k} - {\alpha}_{j,k}}^r}
	\leq 2^{r-1}\Paren{ \bEE{\abs{\widehat{\alpha}_{j,k} - \bar{\alpha}_{j,k}}^r} + \bEE{\abs{\bar{\alpha}_{j,k} - {\alpha}_{j,k}}^r} }\,.
	\]
	The first term can be handled with~\cref{clm:coeff:central}. The second term can be bound using Fact \ref{fact:centralcoeff:rthnorm}. Overall, for $r >2$, we get
	\[
	\bEE{\abs{\widehat{\alpha}_{j,k} - {\alpha}_{j,k}}^r}
	\lesssim \frac{2^{j(r/2-1)}}{m^{r-1}} + \frac{1}{m^{r/2}}
	\leq \frac{2}{m^{r/2}}, 
	\]
since $m \geq 2^j$.
	We get the same bound for $r \in [1,2]$. The result follows.
	The bound on $\bEE{\abs{\widehat{\beta}_{j,k} - {\beta}_{j,k}}^r}$ is obtained in the same way.
\end{proof}
Since, for any $j$, there are $O(2^j)$ translations $k$ for which the coefficients are non-zero (Fact \ref{fact:coeffnonzero}), Claim \ref{clm:coeff} readily implies the corollary below.
\begin{corollary}
	\label{cor:coeff}
	Let $m \geq 2^j$. Then, for $r \geq 1$ and a constant $C$., we have
	\begin{align*}
	\sum_{k \in \Z}\bEE{\abs{\widehat{\alpha}_{j,k} - \alpha_{j,k}}^r} &\leq C\frac{2^j}{m^{r/2}},
\\
	\sum_{k \in \Z}\bEE{\abs{\widehat{\beta}_{j,k} - \beta_{j,k}}^r} &\leq C\frac{2^j}{m^{r/2}}.
	\end{align*}
\end{corollary}

\subsection{Analyzing the error}
For the single-level estimator
\begin{align} 
	\bEE{\norm{\widehat{f}-f}_r^r}
	&= \bEE{\norm{\widehat{f}-f^{(\high)} + f^{(\high)} - f}_r^r}
        \nonumber
        \\
	&\leq 2^{r-1} \paren{ \bEE{\norm{\widehat{f}-f^{(\high)}}_r^r} + \norm{f-f^{(\high)}}_r^r },
        \label{eqn:lossbreakupsl}
\end{align}
where $f^{(\high)} = \sum_{k \in \Z} \alpha_{\high,k} \phi_{\high,k}$. Now, from Fact \ref{fact:centralbias}, we have
$\norm{f-f^{(\high)}}_r^r \lesssim 2^{-\high r \sigma}$, where
\[
\sigma = 
\begin{cases}
	s &\text{if }, r \leq p, \\
	(s-1/p+1/r), &\text{if } r > p.
\end{cases}
\]
Moreover, 
\begin{align*}
	\bEE{\norm{\widehat{f}-f^{(\high)}}_r^r}
	&= \bEE{\norm{\sum_{k \in \Z}(\hat{\alpha}_{\high,k} - \alpha_{\high,k})\phi_{\high,k}}_r^r} \\
	&\lesssim 2^{\high(r/2-1)} \sum_{k\in \Z} \bEE{ |\hat{\alpha}_{\high,k} - \alpha_{\high,k}|^r } &&\paren{\text{Fact \ref{fact:useful:rthnorm}}} \\
	&\lesssim 2^{\high(r/2-1)} \; \frac{2^\high}{m^{r/2}} &&\paren{\text{Corollary \ref{cor:coeff}}} \\
	&= \paren{ \frac{2^\high}{m} }^{r/2}.
\end{align*}
(In our case, $m,H$ will be such that $m \geq 2^\high$ holds, which is why we can use Corollary \ref{cor:coeff}.) Substituting these in \eqref{eqn:lossbreakupsl}, we get
\[
\bEE{\norm{\widehat{f}-f}_r^r} \lesssim 2^{-Hr\sigma} + \paren{ \frac{2^\high}{m} }^{r/2}.
\]
Recall that $m$ is the number of quantized samples available with the referee, where, for a constant $C$,
\[
m = 
\begin{cases}
	n &\text{if } 2^\high \lesssim 2^\ell \text{ (no simulation required)} \\
	Cn2^\ell/2^\high &\text{if } 2^\high \gtrsim 2^\ell \text{ (after simulation)}.
\end{cases}
\]
In other words (ignoring constant $C$) $m = \frac{n2^\ell}{2^\high \vee 2^\ell}$, where $a \vee b = \max\set{a,b}$. Thus,
\[
\bEE{\norm{\widehat{f}-f}_r^r} \lesssim 2^{-Hr\sigma} + \paren{ \frac{2^{\high}(2^\high \vee 2^\ell)}{n2^\ell} }^{r/2} 
\leq \ 2^{-Hr\sigma} + \paren{ \frac{2^{2\high}}{n2^\ell} }^{r/2} + \paren{ \frac{2^{\high}}{n} }^{r/2}.
\]
Setting $H$ such that
\[
2^{\high} = (n2^\ell)^{\frac{1}{2s+2}} \wedge n^{\frac{1}{2s+1}}
\]
gives us
\[
\E\norm{\widehat{f}-f}_r^r \lesssim (n2^\ell)^{-\frac{r\sigma}{2\sigma+2}} \vee n^{-\frac{r\sigma}{2\sigma+1}}.
\]
For $r \leq p$, $\sigma = s$, which proves~\cref{theo:ub:nonadaptive}. \qed

	\section{Analysis of multi-level estimator} \label{supp:multlevel}
		Our goal is to upper bound the worst-case $\cL^r$ loss $\bEE{\norm{\widehat{f}-f}_r^r}$, where $\widehat{f}$ is the estimate output by the referee. We proceed by describing the coupling in the multi-estimator setting.

\subsection{Coupling of simulated and ideal estimators}
Denote by $\bp$ the probability distribution corresponding to density $f$. Recall that we divide the players into $\high - \low + 1$ groups. Suppose the referee obtains $m_J$ (quantized) samples from players in group-$J$. To compute $\bEE{\norm{\widehat{f}-f}_r^r}$, consider the following statistically equivalent situation:
\begin{itemize}
	\item For each $J \in \ibrac{\low,\high}$, there are $m_J$ i.i.d. samples $X_{1,J},\ldots,X_{m_J,J} \sim \bp$.
	
	\item For each $i \in \ibrac{m_\low}$, let
	\begin{equation} \label{eqn:mlcoeff1}
		\widehat{\phi}_{\low,k}(X_{i,\low}) = 
		\begin{cases}
			0 & \text{if } k \notin \cA_{B_{i,\low}}^{(\low)},  \\
			2^{\low/2} Q(V_{i,\low})(k) & \text{if } k \in \cA_{B_{i,\low}}^{(\low)},
		\end{cases}
	\end{equation}
	where $B_{i,\low}$ is the bin (out of $2^\low$ bins) $X_{i,\low}$ lies in, $Q(V_{i,\low})$ is obtained by quantizing $\set{2^{-\low/2}\phi_{\low,k}(X_{i,\low})}_{k \in \cA_{B_{i,\low}}^{(\low)}}$ using Algorithm 1, and $Q(V_{i,\low})(k)$ is the entry in $Q(V_{i,\low})$ corresponding to $k \in \cA_{B_{i,\low}}^{(\low)}$. 
	In other words, $\set{\widehat{\phi}_{\low,k}(X_{i,\low})}_{k \in \Z}$ is the quantized version of $\set{\phi_{\low,k}(X_{i,\low})}_{k \in \Z}$.
	
	\item Similarly, for each $J \in \ibrac{\low,\high}$, for each $i \in \ibrac{m_J}$, let
	\begin{equation} \label{eqn:mlcoeff2}
		\widehat{\psi}_{J,k}(X_{i,J}) = 
		\begin{cases}
			0 & \text{if } k \notin \cB_{B_{i,J}}^{(J)},  \\
			2^{J/2} Q(V_{i,J})(k) & \text{if } k \in \cB_{B_{i,J}}^{(J)}.
		\end{cases}
	\end{equation}
	That is, for each $J \in \ibrac{\low,\high}$, $\set{\widehat{\psi}_{J,k}(X_{i,J})}_{k \in \Z}$ is the quantized version of $\set{\psi_{J,k}(X_{i,J})}_{k \in \Z}$.
	
	\item Given thresholds $\set{t_J}_{J \in \ibrac{\low,\high}}$, define $\widehat{f}$ as 
	\begin{equation} \label{eqn:mlest}
		\widehat{f} = \sum_{k} \widehat{\alpha}_{\low,k} \phi_{\low,k} + \sum_{J=\low}^{\high} \sum_{k} \tilde{\beta}_{J,k} \psi_{J,k},
	\end{equation}
	where 
	\begin{equation} \label{eq:empirical:coeffs:besov}
		\begin{aligned}
			\widehat{\alpha}_{\low,k} &= \frac{1}{m_\low} \sum_{i=1}^{m_\low} \widehat{\phi}_{\low,k} (X_{i,\low}), \\
			\tilde{\beta}_{J,k} &= \widehat{\beta}_{J,k} \indic{|\widehat{\beta}_{J,k}| \geq t_J}
			\text{ with } \widehat{\beta}_{J,k} = \frac{1}{m_J} \sum_{i=1}^{m_J} \widehat{\psi}_{J,k}(X_{i,J}).
		\end{aligned}
	\end{equation}
\end{itemize}
Then, computing $\cL^r$ loss for the multi-level estimator is equivalent to computing $\cL^r$ loss for $\widehat{f}$ defined in \eqref{eqn:mlest}.

\subsection{Setting parameters}
Since we want our multi-level estimator to be adaptive, the parameters $\low, \high$, and $\set{t_J}_{J \in \ibrac{\low,\high}}$ should not depend explicitly on Besov parameters. We set $\low, \high$ as
\begin{align*}
	2^\low  &\eqdef C \Paren{ (\ns 2^\numbits)^{\frac{1}{2(N+1)+2}} \land \ns^{\frac{1}{2(N+1)+1}} },  \\
	2^\high &\eqdef C' \Paren{\frac{\sqrt{\ns 2^\numbits}}{\log \ns 2^\numbits} \land \frac{\ns}{\log^2 \ns}} 
\end{align*}
where $C,C'>0$ are two constants, sufficiently large and small, respectively. Also, note that, since players in group-$J$ have alphabet size $O(2^J)$, we have
\begin{equation*}
	m_J = \frac{\ns}{(\high-\low+1) } \cdot \Paren{\frac{2^\numbits}{2^J}\land 1} \asymp \frac{\ns 2^\numbits}{ H (2^J\lor 2^\numbits)}.
\end{equation*}
Note that the setting of $\high$ implies both $\high 2^\high \ll \sqrt{\ns 2^\numbits}$ and $\high^2 2^\high \ll \ns$, and consequently $m_J \geq J 2^J$.

\paragraph{Threshold values.} How should we set the threshold values $\set{t_J}_{J \in \ibrac{\low,\high}}$?
Since we will pay a cost for the coefficients we zero out (increase in bias), we would like to choose $t_J$ as small as possible. But, in order to have reasonable concentration, we also need $t_J$ to satisfy, for every (sufficiently large) $\gamma > 0$,
\begin{equation*}
	\bPr{\abs{ \widehat{\beta}_{J,k} - \beta_{J,k} } \geq \gamma t_J} \lesssim 2^{-\gamma J}.
\end{equation*}
so that our estimates concentrate well around their true value, and we only zero them out wrongly with very small probability. Now, a natural approach to choose $t_J$ according to the constraint above would be to use Hoeffding's inequality, as $\widehat{\beta}_{J,k}$ is the empirical mean of $m_J$ unbiased estimates of $\beta_{J,k}$, each with magnitude $\lesssim 2^{J/2}$. One can check that this would lead to the setting of $t_J \asymp \sqrt{J 2^J/m_J}$, which, unfortunately, is too big (by a factor of $2^{J/2}$) to give optimal rates.

However, recall that the $m_J$ unbiased estimates, $\widehat{\psi}_{J,k}(X_{i,J})$, are not only such that $|\widehat{\psi}_{J,k}(X_{i,J})|\lesssim 2^{J/2}$; in many cases, they are actually zero, since 
$|\widehat{\psi}_{J,k}(X_{i,J})| \simeq 2^{J/2} \indic{k \in \cB_{B_{{i,J}}}^{(J)}}$. This allows us to derive the following, improving upon the na\"ive use of Hoeffding's inequality.
\begin{lemma}
	\label{lemma:concentration:tJ}
	For $J \in \ibrac{\low,\high}$, setting $t_J \eqdef \sqrt{J/m_J}$, we have
	\[
	\bPr{ \abs{\widehat{\beta}_{J,k}-\beta_{J,k}} \geq \gamma t_J } \leq 2^{-\gamma J}
	\]
	for every $\gamma \geq 6A\norminf{f}$.
\end{lemma}
\begin{proof}
	Fix $J,k$, and consider any $i \in \ibrac{m_J}$. Since $\abs{\widehat{\psi}_{J,k}(X_{i,J})}\leq b \eqdef 2^{J/2}$ and 
	\[
	\bEE{\widehat{\psi}_{J,k}(X_{i,J})^2} = 2^J \bPr{k \in \cB_{B_{i,J}}^{(J)}} \leq 2^J\cdot \norminf{f}\cdot \frac{2A}{2^J} = 2A\norminf{f} \eqdef v
	\]
	where the inequality follows from our assumption that $\supp{\psi}\subseteq [-A,A]$. In particular, we have $v\asymp 1$. Recalling the definition of $\widehat{\beta}_{J,k}$ from~\eqref{eq:empirical:coeffs:besov}, we can apply Bernstein's inequality (\cref{theo:bernstein}) to obtain, for $t \geq 0$ and $\gamma \geq 3v$,
	\[
	\bPr{ \abs{\widehat{\beta}_{J,k}-\beta_{J,k}} \geq \gamma t }
	\leq 
	e^{-\frac{ 3\gamma^2 m_J t^2}{6 v + 2 b \gamma t}} = 
	e^{-\frac{3}{2}\cdot \frac{\gamma m_J t^2}{\frac{3 v}{\gamma} + 2^{J/2} t}}
	\leq e^{-\frac{3}{2}\cdot \frac{\gamma m_J t^2}{1 + 2^{J/2} t}}
	\leq 2^{-\frac{2\gamma m_J t^2}{1 + 2^{J/2} t}}.
	\]
	Setting $t_J \eqdef \sqrt{\frac{J}{m_J}}\lor \frac{J2^{J/2}}{m_J}$, we get $\bPr{ \abs{\widehat{\beta}_{J,k}-\beta_{J,k}} } \leq 2^{-\gamma J}$.  
	Finally, our setting of $\high$ and $m_J$ together imply that $t_J \eqdef \sqrt{\frac{J}{m_J}}$, as (from our choice of parameters) $m_J \geq J 2^J$ for all $J \leq \high$.
\end{proof}

\paragraph{Conclusion.} For constants $C,C',\kappa > 0$, the values of parameters are summarized below.
\begin{align}
	2^\low  &\eqdef C \Paren{ (\ns 2^\numbits)^{\frac{1}{2(N+1)+2}} \land \ns^{\frac{1}{2(N+1)+1}} }  \\
	2^\high &\eqdef C' \Paren{\frac{\sqrt{\ns 2^\numbits}}{\log \ns 2^\numbits} \land \frac{\ns}{\log^2 \ns}}  \\
	m_J &\eqdef \frac{\ns}{(\high-\low+1) } \cdot \Paren{\frac{2^\numbits}{2^J}\land 1} \asymp \frac{\ns 2^\numbits}{ H (2^J\lor 2^\numbits)} \\
	t_J &\eqdef \kappa \sqrt{\frac{J}{m_J}}  
\end{align}
As previously mentioned choices imply both $\high 2^\high \ll \sqrt{\ns 2^\numbits}$ and $\high^2 2^\high \ll \ns$, and consequently $m_J \geq J 2^J$, $J \in \ibrac{\low,\high}$.

\subsection{Analysing the error}

Following the outline of Theorem~3 of~\cite{DonohoJKP96} and Theorem~5.1 of~\cite{ButuceaDKS20}, we will bound $\cL_r$ loss as
\begin{equation}
	\label{eq:bounding:adaptive:error}
	\bEE{\norm{f-\widehat{f}}_r^r} \leq 3^{r-1}\Paren{
		\operatorname{bias}(f)
		+ \operatorname{linear}(f)
		+ \operatorname{details}(f)
	}
\end{equation}
where
\begin{align*}
	\operatorname{bias}(f) &= \bEE{\norm{ f - \sum_{k\in\Z} \alpha_{\high,k}\phi_{\high,k} }_r^r} \\
	\operatorname{linear}(f) &= \bEE{\norm{ \sum_{k\in\Z}  (\widehat{\alpha}_{\low,k}-\alpha_{\low,k})\phi_{\low,k} }_r^r} \\
	\operatorname{details}(f) &= \bEE{\norm{ \sum_{J=\low}^{\high} \sum_{k\in\Z} (\widetilde{\beta}_{J,k}-\beta_{J,k})\psi_{J,k} }_r^r}
\end{align*}
and handle each of the three terms separately. Note that only the third term relates to thresholding. 
\subsubsection{Linear and bias terms}
\textbf{Linear term.} To bound $\operatorname{linear}(f)$, we invoke~\cref{fact:useful:rthnorm,cor:coeff} as in the analysis of single-level estimator. This gives
\begin{align}
	\bEE{ \norm{\sum_{k \in \Z} (\widehat{\alpha}_{\low,k} - \alpha_{\low,k}) \phi_{\low,k}}_r^r }
	&\lesssim 2^{\low(\frac{r}{2}-1)} \sum_{k \in \Z} \bEE{\abs{\widehat{\alpha}_{\low,k} - \alpha_{\low,k}}^r}
	\lesssim 2^{\low(\frac{r}{2}-1)}\cdot\frac{2^{\low}}{m_{\low}^{r/2}} \\
	&= \high^{\frac{r}{2}} \Paren{ \Paren{ \frac{2^{2\low}}{\ns 2^\numbits} }^{r/2} \lor \Paren{ \frac{2^{\low}}{\ns} }^{r/2} } \notag\\
	&\lesssim \high^{\frac{r}{2}} \Paren{ (\ns 2^\numbits)^{-\frac{r(N+1)}{2(N+1)+2}} \lor \ns^{-\frac{r(N+1)}{2(N+1)+1}} } 
	\leq \high^{\frac{r}{2}} \Paren{ (\ns 2^\numbits)^{-\frac{rs}{2s+2}} \lor \ns^{-\frac{rs}{2s+1}} }  \label{eq:linear:term}
\end{align}
where the second-to-last inequality relies on our choice of $\low$.

\textbf{Bias term.} To bound $\operatorname{bias}(f)$, we use Fact \ref{fact:centralbias} to get, with $s' = s-1/p+1/r$,
\begin{equation}
	\label{eq:bias:term}
	\operatorname{bias}(f) \leq C\cdot 2^{-\high s' r}
	\leq C'\cdot
	\Paren{\sqrt{\frac{\log^2(\ns 2^\numbits)}{\ns 2^\numbits}} \lor \frac{\log^2 \ns}{\ns}}^{r(s-1/p+1/r)}.
\end{equation}

\subsubsection{Details term}
To bound the term $\operatorname{details}(f)$, we define, for $J \in \ibrac{\low,\high}$, the three sets of indices:
\begin{align*}
	\widehat{\cI}_{J} &\eqdef \{ k\in\Z : |\widehat{\beta}_{J,k}| > \kappa t_J \} \tag{estimate big: not thresholded}\\
	\cI_{J}^{s} &\eqdef \{ k\in\Z : |\beta_{J,k}| \leq \tfrac{1}{2}\kappa t_J \} \tag{small coefficients}\\
	\cI_{J}^{b} &\eqdef \{ k\in\Z : |\beta_{J,k}| > 2\kappa t_J \} \tag{big coefficients}
\end{align*}
We will partition the error according to these sets of indices, and argue about them separately. Specifically, we write
\begin{align*}
	\operatorname{details}(f) 
	&= \shortexpect\Big[\Big\| \sum_{J=\low}^{\high} \sum_{k\in \widehat{\cI}_{J} \cap \cI_{J}^{s}} (\widetilde{\beta}_{J,k}-\beta_{J,k})\psi_{J,k} \Big\|_r^r\Big]
	+ \shortexpect\Big[\Big\| \sum_{J=\low}^{\high} \sum_{k\in \widehat{\cI}_{J} \setminus \cI_{J}^{s}} (\widetilde{\beta}_{J,k}-\beta_{J,k})\psi_{J,k} \Big\|_r^r\Big] \\
	&\qquad+ \shortexpect\Big[\Big\| \sum_{J=\low}^{\high} \sum_{k\in \cI_{J}^{b}\setminus\widehat{\cI}_{J}} \beta_{J,k}\psi_{J,k} \Big\|_r^r\Big] 
	+ \shortexpect\Big[\Big\| \sum_{J=\low}^{\high} \sum_{k\notin \cI_{J}^{b}\cup\widehat{\cI}_{J}} \beta_{J,k}\psi_{J,k} \Big\|_r^r\Big]\\
	&= E_{bs} + E_{bb} + E_{sb} + E_{ss}, we 
\end{align*}
the four errors coming from the ``big-small,'' ``big-big,'' ``small-big,''  and``small-small'' indices, respectively. Our analysis is along the lines of that in \cite{ButuceaDKS20}.

\paragraph{The term $E_{bs}$.}
We can write 
\begin{align*}
	E_{bs} 
	&\lesssim H^{r/2} \sum_{J=\low}^{\high} 2^{J(\frac{r}{2}-1)} \sum_{k\in\Z} \bEE{ |\widehat{\beta}_{J,k}-\beta_{J,k}|^r \indicSet{k\in\widehat{\cI}_{J} \cap \cI_{J}^{s}}} \tag{\cref{fact:useful:rthnorm}}\\
	&\lesssim H^{r/2} \sum_{J=\low}^{\high} 2^{J(\frac{r}{2}-1)} \sum_{k\in\Z} \bEE{ |\widehat{\beta}_{J,k}-\beta_{J,k}|^{2r}}^{\frac{1}{2}} \bPr{k\in\widehat{\cI}_{J} \cap \cI_{J}^{s}}^{\frac{1}{2}} \tag{Cauchy--Schwarz}\\
	&\lesssim H^{r/2} \sum_{J=\low}^{\high} 2^{J(\frac{r}{2}-1)} \sum_{k\in\Z} \bEE{ |\widehat{\beta}_{J,k}-\beta_{J,k}|^{2r}}^{\frac{1}{2}} \bPr{|\widehat{\beta}_{J,k} - \beta_{J,k}| > \tfrac{\kappa}{2}t_J }\\
	&\lesssim H^{r/2} \sum_{J=\low}^{\high} 2^{J(\frac{r}{2}-1)} \sum_{k\in\Z} \bEE{ |\widehat{\beta}_{J,k}-\beta_{J,k}|^{2r}}^{\frac{1}{2}} 2^{-\frac{\kappa}{2} J} \tag{\cref{{lemma:concentration:tJ}}}\\
	&\lesssim H^{r/2} \sum_{J=\low}^{\high} 2^{J(\frac{r-\kappa}{2}-1)} \frac{2^J}{m_J^{r/2}} 
\end{align*}
where the last inequality follows from~\cref{clm:coeff} and the $O(2^J)$-sparsity of coefficients (Fact \ref{fact:coeffnonzero}). Going forward, recalling our setting of $m_J$ we get
\begin{align}
	\label{eq:ebs}
	E_{bs} 
	&\lesssim H^{r/2} \sum_{J=\low}^{\high}\frac{2^{J\frac{r-\kappa}{2}}}{m_J^{r/2}}
	\lesssim H^{r/2} \Paren{\frac{\high}{\ns 2^\numbits}}^{r/2}\sum_{J=\low}^{\high} 2^{J\frac{r-\kappa}{2}}(2^J\lor 2^\numbits)^{r/2} \notag \\
	&\lesssim H^{r/2} \Paren{\frac{\high}{\ns 2^\numbits}}^{r/2}\sum_{J=\low}^{\high} 2^{J(r-\frac{\kappa}{2})} + H^{r/2} \Paren{\frac{\high}{\ns}}^{r/2}\sum_{J=\low}^{\high} 2^{J\frac{r-\kappa}{2}} \notag \\
	&\lesssim H^{r/2} \Paren{\frac{\high}{\ns 2^\numbits}}^{r/2} 2^{\low (r-\kappa/2)} + H^{r/2} \Paren{\frac{\high}{\ns}}^{r/2} 2^{\low\frac{r-\kappa}{2}}
	\lesssim \high^{r} \ns^{-r/2}
\end{align}
where the inequalities hold for $\kappa > 2r$.

\paragraph{The term $E_{bb}$.}
Turning to the term $E_{bb}$, we have%
\begin{align*}
	E_{bb} &= \Big\| \sum_{J=\low}^{\high} \sum_{k\in \Z} (\widehat{\beta}_{J,k}-\beta_{J,k})\psi_{J,k} \indicSet{k\in \widehat{\cI}_{J} \setminus \cI_{J}^{s}} \Big\|_r^r \\
	&\lesssim H^{r/2} \sum_{J=\low}^{\high} 2^{J(r/2-1)} \sum_{k\in\Z} \bEE{ |\widehat{\beta}_{J,k}-\beta_{J,k}|^r \indicSet{k\in \widehat{\cI}_{J} \setminus \cI_{J}^{s}}} \tag{\cref{fact:useful:rthnorm}}\\
	&\lesssim H^{r/2} \sum_{J=\low}^{\high} \frac{2^{J(r/2-1)}}{m_J^{r/2}} \sum_{k\in\Z} \indicSet{k\notin \cI_{J}^{s}} 
\end{align*}
using that $\indicSet{k\in \widehat{\cI}_{J} \setminus \cI_{J}^{s}}\leq \indicSet{k\notin \cI_{J}^{s}}$ and~\cref{clm:coeff}. Using the definition of $\cI_{J}^{s}$, for any nonnegative sequence $(\alpha_J)_J$ we can further bound this as
\begin{align}
	E_{bb}
	&\lesssim H^{r/2} \sum_{J=\low}^{\high} \frac{2^{J(r/2-1)}}{m_J^{r/2}}  \sum_{k\in\Z} \indicSet{k\notin \cI_{J}^{s}} |\beta_{J,k}|^{\alpha_J} \Paren{\kappa t_J/2}^{-\alpha_J} \notag\\
	&\leq H^{r/2} \sum_{J=\low}^{\high} \frac{2^{J(r/2-1)}}{m_J^{r/2}}  \Paren{\kappa t_J/2}^{-\alpha_J} \sum_{k\in\Z} |\beta_{J,k}|^{\alpha_J}  \notag\\
	&= H^{r/2} \sum_{J=\low}^{\high} 2^{\alpha_J}\kappa^{-\alpha_J} J^{-\frac{\alpha_J}{2}}\frac{2^{J(r/2-1)}}{m_J^{(r-\alpha_J)/2}}  \sum_{k\in\Z} |\beta_{J,k}|^{\alpha_J}  \notag\\
	&\leq H^{r/2} \sum_{J=\low}^{\high} \frac{2^{J(r/2-1)}}{m_J^{(r-\alpha_J)/2}}  \sum_{k\in\Z} |\beta_{J,k}|^{\alpha_J} \label{eq:ebb:reusable:for:ess}
\end{align}
the last inequality using $\kappa/2\geq 1$ and $J\geq 1$ to simplify the expression a little. For now, we ignore the factor $H^{r/2}$ (we will bring it back at the end), and look at two cases:
\begin{itemize}
	\item If $p > \frac{r}{s+1}$, we continue by writing
	\begin{align*}
		E_{bb} &\lesssim \sum_{J=\low}^{\high} \frac{2^{J(r/2-1)}}{m_J^{(r-\alpha_J)/2}}  2^{-J\alpha_J(s+\frac{1}{2}-\frac{1}{\alpha_J})} 
		\asymp \sum_{J=\low}^{\high} m_J^{-\frac{r-\alpha_J}{2}} 2^{\frac{1}{2}J(r-(2s+1)\alpha_J)}
		= \sum_{J=\low}^{\high} 2^{J\frac{r}{2}} m_J^{-\frac{r}{2}} \cdot 2^{\frac{\alpha_J}{2}(\log m_J - J(2s+1))}
	\end{align*}

	where the first inequality, which holds for any $\alpha_J\in[0,p]$, uses the following bound on $\sum_k|\beta_{J,k}|^p$: 
	For any $\alpha\in[0,p]$, we have from~\cref{fact:useful:bound:beta:norm} and H\"older's inequality along with the sparsity of coefficients (Fact \ref{fact:coeffnonzero}), that 
	\[
	\sum_{k\in\Z} |\beta_{J,k}|^{\alpha} = \sum_{k\in\Z} |\beta_{J,k}|^{\alpha}\indic{\beta_{J,k}\neq 0} 
	\leq \Big( \sum_{k\in\Z} |\beta_{J,k}|^{p} \Big)^{\frac{\alpha}{p}}\abs{\cB_J}^{1-\frac{\alpha}{p}}
	\leq C^{\frac{\alpha}{p}} (2A+1)^{1-\frac{\alpha}{p}} \cdot 2^{-J\alpha(s+\frac{1}{2}-\frac{1}{p})} \cdot 2^{J(1-\frac{\alpha}{p})}
	\]
	so that $\sum_{k\in\Z} |\beta_{J,k}|^{\alpha}\lesssim 2^{-J\alpha(s+\frac{1}{2}-\frac{1}{\alpha})}$, as in~\cite[Section~C.2.3]{ButuceaDKS20}. To bound the resulting sum, we need to choose $\alpha_J\in[0,p]$ for all $J$ in order to minimize the result. Since $m_J \propto 2^{J-\numbits}\land 1$, the quantity
	\[
	\log m_J - J(2s+1)
	\]
	is decreasing in $J$, and thus becomes negative at some value $\medium$ (for simplicity, assumed to be an integer), such that\footnote{To see why, recall that \[
		\log m_J - J(2s+1) = \log\frac{\ns}{\high} - (J-\numbits)_+ - J(2s+1) + O(1)
		\]
		from our setting of $m_J$. Finding the value of $J$ for which $\log\frac{\ns}{\high} - (J-\numbits)_+ - J(2s+1)$ cancels gives the claimed relation.}
	\begin{equation}
		\label{eq:setting:M}
		2^{\medium} \asymp  \Paren{ \frac{\ns 2^\numbits}{\high } }^{\frac{1}{2s+2}} \land \Paren{ \frac{\ns}{\high} }^{\frac{1}{2s+1}}
	\end{equation}
	we see that we should set $\alpha_J \eqdef 0$ for $J\leq \medium$, and for $J>M$ set all $\alpha_J$ to some value $\alpha=\alpha(r,s)$ which will balance the remaining terms. With this choice, we can write
	\begin{align*}
		\sum_{J=\low}^{\high} m_J^{-\frac{r-\alpha_J}{2}} 2^{\frac{1}{2}J(r-(2s+1)\alpha_J)}
		&\leq \sum_{J=\low}^{\medium} m_J^{-\frac{r}{2}} 2^{J\frac{r}{2}} + \sum_{J=\medium}^{\high} m_J^{-\frac{r-\alpha}{2}} 2^{\frac{1}{2}J(r-(2s+1)\alpha)}  \\
		&\leq  
		\Paren{\frac{\high}{\ns 2^\numbits}}^{\frac{r}{2}} \sum_{J=1}^{\medium}2^{Jr} 
		+ \Paren{\frac{\high}{\ns}}^{\frac{r}{2}} \sum_{J=1}^{\medium}2^{J\frac{r}{2}} \\
		&\qquad+ \Paren{\frac{\high}{\ns 2^\numbits}}^{\frac{r-\alpha}{2}} \sum_{J=\medium}^{\high} 2^{J(r-(s+1)\alpha)}
		+ \Paren{\frac{\high}{\ns}}^{\frac{r-\alpha}{2}}\sum_{J=\medium}^{\high} 2^{\frac{1}{2}J(r-(2s+1)\alpha)}
	\end{align*}
	recalling for the second inequality that $m_J^{-1} \asymp \frac{\high}{\ns}(\frac{2^J}{2^\numbits}\lor 1)$. 
	We can bound the first and second terms as
	\[
	\Paren{\frac{\high}{\ns 2^\numbits}}^{\frac{r}{2}} \sum_{J=1}^{\medium}2^{Jr}
	\leq \frac{2^r}{2^r - 1}\Paren{\frac{\high}{\ns 2^\numbits}}^{\frac{r}{2}} 2^{r\medium}, \qquad 
	\Paren{\frac{\high}{\ns}}^{\frac{r}{2}} \sum_{J=1}^{\medium}2^{J\frac{r}{2}}
	\leq \frac{2^{r/2}}{2^{r/2} - 1}\Paren{\frac{\high}{\ns}}^{\frac{r}{2}} 2^{\frac{r}{2}\medium}
	\]
	and from~\eqref{eq:setting:M} we get that their sum is then
	\[
	\Paren{\frac{\high}{\ns 2^\numbits}}^{\frac{r}{2}} \sum_{J=1}^{\medium}2^{Jr} + \Paren{\frac{\high}{\ns}}^{\frac{r}{2}} \sum_{J=1}^{\medium}2^{J\frac{r}{2}}
	\lesssim  \Paren{\frac{\high}{\ns 2^\numbits}}^{\frac{r}{2}} 2^{r\medium} \lor \Paren{\frac{\high}{\ns}}^{\frac{r}{2}} 2^{\frac{r}{2}\medium}
	\lesssim \Paren{\frac{\high}{\ns 2^\numbits}}^{\frac{rs}{2s+2}} \lor \Paren{\frac{\high}{\ns}}^{\frac{rs}{2s+1}}\,.
	\]
	Thus, it only remains to handle the third and fourth terms by choosing a suitable value for $\alpha$.
	Recalling that we are in the case $p > \frac{r}{s+1}$, we pick any $\frac{r}{s+1} < \alpha \leq p$; for instance, $\alpha\eqdef p$. Since  $r-(s+1)p < 0$ we then have
	\begin{align*}
		\Paren{\frac{\high}{\ns 2^\numbits}}^{\frac{r-p}{2}}  \sum_{J=\medium}^{\high} 2^{J(r-(s+1)p)}
		\leq \Paren{\frac{\high}{\ns 2^\numbits}}^{\frac{r-p}{2}} \frac{2^{\medium(r-(s+1)p)}}{1-2^{r-(s+1)p}}
		\asymp \Paren{\frac{\high}{\ns 2^\numbits}}^{\frac{r-p}{2}} 2^{\medium(r-(s+1)p)}\,;
	\end{align*}
	note that $\frac{1}{1-2^{r-(s+1)p}}>0$ is a constant, depending only on $r,s,p$. Similarly, $r-(2s+1)p < 0$, and so
	\[
	\Paren{\frac{\high}{\ns}}^{\frac{r-p}{2}}\sum_{J=\medium}^{\high} 2^{\frac{1}{2}J(r-(2s+1)p)}
	\leq \Paren{\frac{\high}{\ns}}^{\frac{r-p}{2}} \frac{2^{\frac{1}{2}\medium(r-(2s+1)p)}}{1-2^{\frac{1}{2}(r-(2s+1)p)}}
	\asymp \Paren{\frac{\high}{\ns}}^{\frac{r-p}{2}} 2^{\frac{1}{2}\medium(r-(2s+1)p)}.
	\]
	From the setting of $\medium$ from~\eqref{eq:setting:M}, by a distinction of cases we again can bound their sum as
	\[
	\Paren{\frac{\high}{\ns 2^\numbits}}^{\frac{r-p}{2}}  \sum_{J=\medium}^{\high} 2^{J(r-(s+1)p)}
	+ \Paren{\frac{\high}{\ns}}^{\frac{r-p}{2}}\sum_{J=\medium}^{\high} 2^{\frac{1}{2}J(r-(2s+1)p)}
	\lesssim \Paren{\frac{\high}{\ns 2^\numbits}}^{\frac{rs}{2s+2}} \lor \Paren{\frac{\high}{\ns}}^{\frac{rs}{2s+1}}\,.
	\]
	Therefore, overall, in the case $p > \frac{r}{s+1}$ we have (bringing back the factor $H^{r/2}$ we had ignored earlier)
	\begin{equation}
		\label{eq:ebb:1}
		E_{bb} \lesssim H^{r/2} \Paren{\frac{\high}{\ns 2^\numbits}}^{\frac{rs}{2s+2}} \lor H^{r/2} \Paren{\frac{\high}{\ns}}^{\frac{rs}{2s+1}}\,.
	\end{equation}

	\item If $p \leq \frac{r}{s+1}$, we will choose $\alpha_J \geq p$ for all $J$. Under this constraint, we can use the monotonicity of $\lp[p]$ norms (for every $x$, $\norm{x}_p \leq \norm{x}_q$ if $p\geq q$) to write
	\begin{align*}
		E_{bb}
		&\lesssim \sum_{J=\low}^{\high} \frac{2^{J(r/2-1)}}{m_J^{(r-\alpha_J)/2}}  \sum_{k\in\Z} |\beta_{J,k}|^{\alpha_J} 
		\leq \sum_{J=\low}^{\high} \frac{2^{J(r/2-1)}}{m_J^{(r-\alpha_J)/2}} \Paren{ \sum_{k\in\Z} |\beta_{J,k}|^{p} }^{\alpha_J/p} \\
		&\lesssim \sum_{J=\low}^{\high} \frac{2^{J(r/2-1)}}{m_J^{(r-\alpha_J)/2}} 2^{-J\alpha_J(s+\frac{1}{2}-\frac{1}{p})} \tag{\cref{fact:useful:bound:beta:norm}} \\
		&= \sum_{J=\low}^{\high} m_J^{-\frac{r-\alpha_J}{2}} 2^{J(\frac{r}{2}-1-\alpha_J(s+\frac{1}{2}-\frac{1}{p}))}.
	\end{align*}
	As before, one can see that for there exists some $\medium$ such that the best choice is to set $\alpha_J = p$ for $J\leq \medium$ (as small as possible given our constraint $\alpha_J \geq p$). Moreover, proceeding as in the previous case,\footnote{That is, find the value $J$ solving (approximately) the equation $\log\frac{\ns}{\high} - (J-\numbits)_+ - J(2s+1-2/p)=0$ (note that the LHS is again decreasing in $J$).} we can see that this $\medium$ is such that
	\begin{equation}
		\label{eq:setting:M:2}
		2^{\medium} \asymp  \Paren{ \frac{\ns 2^\numbits}{\high } }^{\frac{1}{2(s-1/p)+2}} \land \Paren{ \frac{\ns}{\high} }^{\frac{1}{2(s-1/p)+1}}.
	\end{equation}
	
	This part of the sum will then contribute
	\begin{align*}
		\sum_{J=\low}^{\medium} m_J^{-\frac{r-p}{2}} 2^{J(\frac{r}{2}-1-p(s+\frac{1}{2}-\frac{1}{p}))}
		& \asymp\Paren{\frac{\high}{\ns 2^\numbits}}^{\frac{r-p}{2}}\sum_{J=\low}^{\medium} 2^{J(r-1-p(s+1-\frac{1}{p}))}
		+ \Paren{\frac{\high}{\ns}}^{\frac{r-p}{2}}\sum_{J=\low}^{\medium} 2^{J(\frac{r}{2}-1-p(s+\frac{1}{2}-\frac{1}{p}))} \\
		&\asymp \Paren{\frac{\high}{\ns 2^\numbits}}^{\frac{r-p}{2}}2^{\medium(r-p(s+1))} 
		+ \Paren{\frac{\high}{\ns}}^{\frac{r-p}{2}}2^{\medium(\frac{r}{2}-1-p(s+\frac{1}{2}-\frac{1}{p}))} \\
		&\asymp \Paren{\frac{\high}{\ns 2^\numbits}}^{\frac{r(s-1/p+1/r)}{2(s-1/p)+2}} \lor \Paren{\frac{\high}{\ns}}^{\frac{r(s-1/p+1/r)}{2(s-1/p)+1}}.
	\end{align*}
	For $J>\medium$, we choose an arbitrary constant $\alpha \geq p$ such that $\alpha > \frac{r-1}{s+1-1/p}$ (so that $r-1-\alpha(s+1-1/p) < 0$), and set $\alpha_J=\alpha$ for all $J>\medium$.\footnote{It will be important later, when bounding $E_{ss}$, to note that $\frac{r-1}{s+1-1/p}<r$, and thus one can also enforce $\alpha \leq r$.} Observe that this implies $\alpha > \frac{r/2-1}{s+1/2-1/p}$. This part of the sum will then contribute at most
	\begin{align*}
		\sum_{J=\medium}^{\high} m_J^{-\frac{r-\alpha}{2}} 2^{J(\frac{r}{2}-1-\alpha(s+\frac{1}{2}-\frac{1}{p}))}
		&\asymp \Paren{\frac{\high}{\ns 2^\numbits}}^{\frac{r-\alpha}{2}}\sum_{J=\medium}^{\high} 2^{J(r-1-\alpha(s+1-\frac{1}{p}))}
		+ \Paren{\frac{\high}{\ns}}^{\frac{r-\alpha}{2}}\sum_{J=\medium}^{\high} 2^{J(\frac{r}{2}-1-\alpha(s+\frac{1}{2}-\frac{1}{p}))} \\
		&\asymp \Paren{\frac{\high}{\ns 2^\numbits}}^{\frac{r-\alpha}{2}} 2^{\medium(r-1-\alpha(s+1-\frac{1}{p}))}
		+ \Paren{\frac{\high}{\ns}}^{\frac{r-\alpha}{2}}2^{\medium(\frac{r}{2}-1-\alpha(s+\frac{1}{2}-\frac{1}{p}))} \\
		&\asymp \Paren{\frac{\high}{\ns 2^\numbits}}^{-\frac{r(s-1/p+1/r)}{2(s-1/p)+2}}\lor \Paren{\frac{\high}{\ns}}^{\frac{r(s-1/p+1/r)}{2(s-1/p)+1}}
	\end{align*}
	as well. Thus, overall, in the case $p \leq \frac{r}{s+1}$ we have (bringing back the factor $H^{r/2}$ we had ignored earlier)
	\begin{equation}
		\label{eq:ebb:2}
		E_{bb} \lesssim H^{r/2} \Paren{\frac{\high}{\ns 2^\numbits}}^{\frac{r(s-1/p+1/r)}{2(s-1/p)+2}}\lor H^{r/2} \Paren{\frac{\high}{\ns}}^{\frac{r(s-1/p+1/r)}{2(s-1/p)+1}}\,.
	\end{equation}
\end{itemize}

\paragraph{The term $E_{sb}$.} To handle the term $E_{sb}$, we will rely on the fact that, for any $r\geq p$, we have the inclusion $\besov(p,q,s)\subseteq \besov(r,q,s')$, for $s' = s - \Paren{\frac{1}{p} - \frac{1}{r}}$. This will let us use~\cref{fact:useful:bound:beta:norm} on $\sum_{k\in\Z} |\beta_{J,k}|^r$:
\begin{align}
	E_{sb}
	&\lesssim H^{r/2} \sum_{J=\low}^{\high} 2^{J(\frac{r}{2}-1)} \sum_{k\in\Z} \bEE{ |\beta_{J,k}|^r \indicSet{k\in\cI_{J}^{b}\setminus \widehat{\cI}_{J}}} \tag{\cref{fact:useful:rthnorm}}\\ 
	&\lesssim H^{r/2} \sum_{J=\low}^{\high} 2^{J(\frac{r}{2}-1)} \sum_{k\in\Z} |\beta_{J,k}|^r \bPr{|\widehat{\beta}_{J,k} - \beta_{J,k}| > \kappa t_J } \notag\\ 
	&\lesssim H^{r/2} \sum_{J=\low}^{\high} 2^{J(\frac{r}{2}-1)} \sum_{k\in\Z} |\beta_{J,k}|^r 2^{-\kappa J} \tag{\cref{{lemma:concentration:tJ}}}\\ 
	&\lesssim H^{r/2} \sum_{J=\low}^{\high} 2^{J(\frac{r}{2}-1-\kappa)} 2^{-J r(s'+\frac{1}{2}-\frac{1}{r})} \tag{\cref{fact:useful:bound:beta:norm}}\\
	&= H^{r/2} \sum_{J=\low}^{\high} 2^{-J(r s'+\kappa) }  \leq H^{r/2} \frac{2^{-\low(r s'+\kappa) }}{1-2^{-(r s'+\kappa)}} \notag\\
	&\lesssim H^{r/2} 2^{-\low r (N+1)} \lesssim H^{r/2} (\ns 2^\numbits)^{-\frac{r(N+1)}{2(N+1)+2}} \lor H^{r/2} \ns^{-\frac{r(N+1)}{2(N+1)+1}}
	\leq H^{r/2} (\ns 2^\numbits)^{-\frac{rs}{2s+2}} \lor H^{r/2} \ns^{-\frac{rs}{2s+1}} \label{eq:esb}
\end{align}
where for the third-to-last inequality we relied on our choice of $\kappa \geq r(N+1)$, and for the second-to-last, on our setting of $\low$.
\paragraph{The term $E_{ss}$.} 

Finally, we bound the last error term for $\operatorname{details}(f)$, $E_{ss}$. In view of proceeding as for $E_{bb}$, for any nonnegative sequence $(\alpha_J)_J$ with $0\leq \alpha_J \leq r$, we can write
\begin{align*}
	E_{ss}
	&\lesssim H^{r/2} \sum_{J=\low}^{\high} 2^{J(\frac{r}{2}-1)} \sum_{k\in\Z} \bEE{ |\beta_{J,k}|^r \indicSet{k\in\cI_{J}^{s}\setminus \widehat{\cI}_{J}}} \tag{\cref{fact:useful:rthnorm}}\\ 
	&\leq H^{r/2} \sum_{J=\low}^{\high} 2^{J(\frac{r}{2}-1)} \sum_{k\in\Z} |\beta_{J,k}|^r \indicSet{k\in\cI_{J}^{s}}\\ 
	&\leq H^{r/2} \sum_{J=\low}^{\high} 2^{J(\frac{r}{2}-1)} \sum_{k\in\Z} |\beta_{J,k}|^{\alpha_J} (\kappa t_J/2)^{r-\alpha_J} \indicSet{k\in\cI_{J}^{s}}\\ 
	&\leq H^{r/2} (\kappa/2)^r\sum_{J=\low}^{\high} 2^{J(\frac{r}{2}-1)}  J^{\frac{r-\alpha_J}{2}} m_J^{-(r-\alpha_J)/2} \sum_{k\in\Z} |\beta_{J,k}|^{\alpha_J}\\ 
	&\leq (\kappa/2)^r \high^{r}\sum_{J=\low}^{\high} \frac{2^{J(r/2-1)}}{m_J^{(r-\alpha_J)/2}} \sum_{k\in\Z} |\beta_{J,k}|^{\alpha_J}
\end{align*}
which is, except for the extra factor of $(\kappa/2)^r \high^{\frac{r}{2}}$, exactly the same expression as~\eqref{eq:ebb:reusable:for:ess}. We can thus continue the analysis of $E_{ss}$ the same way as we did $E_{bb}$, noting that since $r\geq p$ all the choices for $\alpha_J$ in that analysis are still possible; leading to the bound:
\begin{equation}
	\label{eq:ess}
	E_{ss} \lesssim 
	\begin{cases}
		\high^{r}\cdot\Paren{\Paren{\frac{\high}{\ns 2^\numbits}}^{\frac{rs}{2s+2}}\lor \Paren{\frac{\high}{\ns}}^{\frac{rs}{2s+1}} }  & p > \frac{r}{s+1}\\
		\high^{r}\cdot\Paren{ \Paren{\frac{\high}{\ns 2^\numbits}}^{\frac{r(s-1/p+1/r)}{2(s-1/p)+2}}\lor \Paren{\frac{\high}{\ns}}^{\frac{r(s-1/p+1/r)}{2(s-1/p)+1}} }& p \leq \frac{r}{s+1}.
	\end{cases}
\end{equation}

\subsubsection{Total error} 
Defining, for $r\geq 1$, $s\geq 0$, and $p\geq 1$, the quantities
\[
\nu(r,p,s) \eqdef \frac{rs}{2s+2} \indicSet{p > \frac{r}{s+1}} + \frac{r(s-1/p+1/r)}{2(s-1/p)+2} \indicSet{p\leq \frac{r}{s+1}}
\]
and
\[
\mu(r,p,s) \eqdef \frac{rs}{2s+1} \indicSet{p > \frac{r}{s+1}} + \frac{r(s-1/p+1/r)}{2(s-1/p)+1} \indicSet{p\leq \frac{r}{s+1}}
\]
we can gather all the error terms from~\cref{eq:linear:term,eq:bias:term,eq:ebs,eq:ebb:1,eq:ebb:2,eq:esb,eq:ess}, to get
\[
\bEE{\norm{f-\widehat{f}}_r^r} 
\lesssim \high^\kappa\Paren{ (\ns 2^\numbits)^{-\frac{r(s-1/p+1/r)}{2}}\lor \ns^{-r(s-1/p+1/r)} + (\ns 2^\numbits)^{-\frac{rs}{2s+2}}\lor \ns^{-\frac{rs}{2s+1}} + (\ns 2^\numbits)^{-\nu(r,p,s)}\lor \ns^{-\mu(r,p,s)} }
\]
where $\kappa=\kappa(s,r,p)$ is a constant obtained for simplicity by taking the maximum of the exponent of $\high$ in the previous bounds. %
To simplify this expression, we observe that the following holds for all $p,r,s\geq 1$:
\begin{itemize}
  \item $\nu(r,p,s) \leq \frac{r(s-1/p+1/r)}{2}$
  \item $\nu(r,p,s) \leq \frac{rs}{2s+2}$
  \item $\mu(r,p,s) \leq r\Paren{s-\frac{1}{p}+\frac{1}{r}}$
  \item $\mu(r,p,s) > \frac{rs}{2s+1}$ if, and only if, $r \in ((s+1)p,(2s+1)p)$  (and of course $\mu(r,p,s) = \frac{rs}{2s+1}$ if $r \leq (s+1)p$)
\end{itemize}
(this follows from somewhat tedious algebraic manipulations and distinctions of cases). Given the above, we finally get the following bound (where we loosened the bound on the exponent of $\high$ to make the result simpler to state):
\begin{align}
	\bEE{\norm{f-\widehat{f}}_r^r} 
	&\lesssim \high^{\kappa} \Paren{ (\ns 2^\numbits)^{-\nu(r,p,s)}\lor \ns^{-\mu(r,p,s)} \lor \ns^{-\frac{rs}{2s+1}} } \notag\\
	&= \log^{\kappa} \ns \cdot \begin{cases}
	 (\ns 2^\numbits)^{-\frac{rs}{2s+2}}\lor \ns^{-\frac{rs}{2s+1}} &\text{ if } r < (s+1)p \\
	 (\ns 2^\numbits)^{-\frac{r(s-1/p+1/r)}{2(s-1/p)+2}}\lor \ns^{-\frac{rs}{2s+1}} &\text{ if } (s+1)p \leq r < (2s+1)p \\
	 (\ns 2^\numbits)^{-\frac{r(s-1/p+1/r)}{2(s-1/p)+2}}\lor \ns^{-\frac{r(s-1/p+1/r)}{2(s-1/p)+1}} &\text{ if } r \geq (2s+1)p
	\end{cases} 
\end{align}
which proves Theorem 1.3 in the main paper.\qed

	\section{Proof of lower bounds} \label{supp:lowerbound}
		Our lower bound construction will depend on whether $r < (s+1)p$ or $r \geq (s+1)p$. Before delving into these cases, we first (a)~recall the result from \cite{AcharyaCT20} that we will use to upperbound average discrepancy; (b)~discuss how the consideration of binary hypothesis testing problem gives a lower bound on average discrepancy.

\subsection{Upper bound on average discrepancy} \label{sec:divupbound}
Consider the following assumptions on $\cP = \{\p_z : z \in \set{-1,1}^d\}$, where $\p_z$'s are probability distributions on $[0,1]$.

\begin{assumption}[Densities exist]
	\label{assn:1}
	For every $z \in \set{-1,1}^d$ and $i \in \ibrac{d}$, there exist functions $\phi_{z,i}:[0,1] \to \R$ such that $\E_{\p_z}[\phi_{z,i}^2] = 1$ and 
	\[
	\frac{d\p_{z^{\oplus i}}}{d\p_z} = 1 + \alpha \phi_{z,i}
	\]
	where $\alpha \in \R$ is a fixed constant independent of $z,i$.
\end{assumption}

\begin{assumption}[Boundedness of ratios of message probabilities]
	\label{assn:2}
	There exists some $\lambda \in [1,\infty]$ such that  
	\[
	\max_{z \in \set{-1,1}^d, y \in \set{0,1}^\numbits} \ \sup_{W \in \cW} 
	\frac{ \E_{\p_{z^{\oplus i}}} [W(y|X)] }{ \E_{\p_z} [W(y|X)] } \leq \lambda
	\]
	where $\cW = \set{W\colon[0,1] \to \set{0,1}^\numbits}$ is the collection of all $\numbits$-bit channels.
\end{assumption} In particular, if $\frac{d\p_{z^{\oplus i}}}{d\p_z}$ are uniformly bounded (over $z,i$) by $\lambda'$, then this assumption is satisfied with $\lambda = \lambda'$ \cite{AcharyaCT20}.

\begin{assumption}[Orthonormality]
	\label{assn:3}
	For all $z \in \set{-1,1}^d$ and $i,j \in \ibrac{d}$, $\E_{\p_z}[\phi_{z,i} \phi_{z,j}] = \indic{i=j}$.
\end{assumption}

\begin{theorem} [\cite{AcharyaCT20}] \label{thm:AcharyaCT20}
	Suppose $\cP$ satisfies~\cref{assn:1,assn:2,assn:3}. For some $\tau \in (0,1/2]$, let $\pi$ be a prior on $Z \in \set{-1,1}^d$ defined as $Z_i \sim {\tt Rademacher}(\tau)$ independently for each $i \in \ibrac{d}$. For $Z \sim \pi$, let $X_1,\ldots,X_n$ be i.i.d. samples from $\p_Z$. Then, for any interactive protocol generating $\numbits$-bit messages $Y_1,\ldots,Y_n$, we have
	\[
	\paren{ \frac{1}{d}\sum_{i=1}^{d}\totalvardist{\p^{Y^n}_{-i}}{\p^{Y^n}_{+i}} }^2
	\leq
	\frac{2}{d} (\lambda \wedge \tau^{-1}) n \alpha^2 2^\numbits.
	\]
\end{theorem}

\subsection{Lower bound on average discrepancy} \label{sec:divlowbound}
For $Z \sim \pi$, let $X_1,\ldots,X_n$ be i.i.d.\ samples from $\p_Z$ distributed across $n$ players, and let $Y_1,\ldots,Y_n$ be $\numbits$-bit messages sent by the players (possibly interactively) to the referee. Based on the $\numbits$-bit messages $Y_1,\ldots,Y_n$, suppose the referee outputs an estimate $\hat{Z} = (\hat{Z}_1,\ldots,\hat{Z}_d)$ of $Z = (Z_1,\ldots,Z_d)$. Then, an upper bound on $\sum_{i=1}^{d}\Pr\set{\hat{Z}_i \neq Z_i}$ gives a lower bound on $\sum_{i=1}^{d}\kldiv{\p^{Y^n}_{-i}}{\p^{Y^n}_{+i}}$. To see this, note that, for a given $i \in \ibrac{d}$,
\begin{align*}
	\Pr\set{\hat{Z}_i \neq Z_i} 
	&= \Pr\set{\hat{Z}_i = -1 | Z_i = 1}\Pr\set{Z_i=1} + \Pr\set{\hat{Z}_i =1 | Z_i = -1}\Pr\set{Z_i=-1} \\
	&= \tau \paren{1 - \Pr\set{\hat{Z}_i = 1 | Z_i = 1}} + (1-\tau)\Pr\set{\hat{Z}_i =1 | Z_i = -1} \\
	&\geq \tau \paren{1 - \Pr\set{\hat{Z}_i = 1 | Z_i = 1}} + \tau \Pr\set{\hat{Z}_i =1 | Z_i = -1} \tag{since $(1-\tau)\geq \tau$ for $\tau \leq 1/2$} \\
	&= \tau\paren{1 - \paren{\Pr\set{\hat{Z}_i = 1 | Z_i = 1} - \Pr\set{\hat{Z}_i = 1 | Z_i = -1}} } \\
	&\geq \tau \paren{ 1 - \totalvardist{\p^{Y^n}_{+i}}{\p^{Y^n}_{-i}} }.
\end{align*}
Thus, 
\begin{align*}
	\sum_{i=1}^{d} \Pr\set{\hat{Z}_i \neq Z_i} 
	\geq \tau \paren{ d - \sum_{i=1}^{d}\totalvardist{\p^{Y^n}_{+i}}{\p^{Y^n}_{-i}} }
\end{align*}
which gives
\begin{equation} \label{eqn:dtvlb}
	\frac{1}{d}\sum_{i=1}^{d}\totalvardist{\p^{Y^n}_{+i}}{\p^{Y^n}_{-i}}
	\geq 1 - \frac{1}{d\tau}\sum_{i=1}^{d} \Pr\set{\hat{Z}_i \neq Z_i}.
\end{equation}
In conclusion, to get a lower bound on average discrepancy, it suffices to upperbound 
$\sum_{i=1}^{d} \Pr\set{\hat{Z}_i \neq Z_i}$ for an estimator $\hat{Z}$ of $Z$.

\subsection{Lower bound on $\cL_r^*(\ns, \numbits, p,q,s)$ for $r < (s+1)p$} \label{sec:lb1}
{\bf Construction.} The family of distributions $\cP_1$ that we will use to derive lower bound when $r < (s+1)p$ has also been used in deriving lower bounds in the unconstrained setting \cite{DonohoJKP96,HardleKPT12} and in the LDP setting \cite{ButuceaDKS20}. \medskip

Let $g_0$ be a density function (see~\cite[p.157]{HardleKPT12}) such that 
\begin{enumerate}
	\item ${\rm supp}(g_0) \subseteq [0,1]$;
	\item $\norm{g_0}_{pqs} \leq 1/2$;
	\item $g_0 \equiv c_0 > 0$ on some interval $[a,b] \subseteq [0,1]$.
\end{enumerate}
In what follows, $j$ is a free parameter that will be suitably chosen later in the proof. Let $\psi_{j,k}$ be defined as $\psi_{j,k}(x) = 2^{j/2}\psi(2^jx - k)$, where $\psi$ is the mother wavelet used to define $\norm{\cdot}_{pqs}$ (see~\cref{sec:besov}). It is a fact that $\int \psi_{j,k}(x) dx = 0$ for every $j,k$ \cite{HardleKPT12}. \medskip

For a given $z \in \set{-1,1}^d$, define
\begin{equation}
	f_z \eqdef g_0 + \gamma \sum_{k \in \cI_j}z_k \psi_{j,k}
\end{equation}
where 
\begin{itemize}
	\item $\cI_j$ is the set of indices $k \in \Z$ such that
	\begin{itemize}
		\item[i.] ${\rm supp}(\psi_{j,k}) \subseteq [a,b]$ for every $k \in \cI_j$;
		\item[ii.] for $k,k' \in \cI_j$, $k \neq k'$, $\psi_{j,k}$ and $\psi_{jk'}$ have disjoint support;
		\item[iii.] $d\eqdef \abs{\cI_j} =  C 2^j$, for a constant $C$. Here on, we will assume for simplicity that $d = 2^j$.
	\end{itemize}	
	\item $\gamma$ is chosen such that
	\begin{itemize}
		\item[i.] $f_z(x) \geq c_0/2$ for every $x \in [a,b]$; this condition is satisfied if $c_0 - \gamma 2^{j/2}\norm{\psi}_\infty \geq c_0/2$, \ie $\gamma \leq (c_0/2\norm{\psi}_\infty) 2^{-j/2}$.
		\item[ii.] $\norm{f_z}_{pqs} \leq 1$; since $\norm{f_z}_{pqs} \leq \norm{g_0}_{pqs} + \gamma \norm{\psi_{j,k}}_{pqs} \leq 1/2 + \gamma C 2^{j/p} 2^{j(s+1/2-1/p)}$ (see pg. 160 in \cite{HardleKPT12}), we get that $\norm{f_z}_{pqs} \leq 1$ if $\gamma \leq (1/2C) 2^{-j(s+1/2)}$.
	\end{itemize}
	Since $s > 1/p > 0$, we get that for $j$ large enough, if $\gamma$ satisfies condition (ii), it automatically satisfies condition (i). Thus, we choose $\gamma = C 2^{-j(s+1/2)}$ for some constant $C$.
\end{itemize}
Finally, we define the family of distributions as
\begin{equation}
	\cP_1 = \set{\p_z : \p_z \text{ has density } f_z = g_0 + \gamma \sum_{k \in \cI_j}z_k \psi_{j,k}, \ z \in \set{-1,1}^d }.
\end{equation}

\paragraph{Prior on $Z$.} We assume a uniform prior on $Z \in \set{-1,1}^d$, \ie $Z_i \sim {\tt Rademacher}(1/2)$ independently for each $i \in \ibrac{d}$.

\paragraph{Upper bound on average discrepancy.} To upperbound average discrepancy, we verify that $\cP_1$ satisfies the three assumptions described in~\cref{sec:divupbound}, and then use \cref{thm:AcharyaCT20}. For any $z \in \set{-1,1}^d$, $k \in \ibrac{d}$, we have
\[
\frac{d\p_{z^{\oplus k}}}{d\p_z}(x) = 1 - \frac{2\gamma z_k \psi_{j,k}(x)}{c_0 + \gamma z_k \psi_{j,k}(x)}.
\]
Since ${\rm supp}(\psi_{j,k}) \cap {\rm supp}(\psi_{j,k'})$ is empty for $k \neq k'$, it follows that~\cref{assn:1,assn:3} hold. Moreover, from our setting of $\gamma$, we have that, for every $x$, $c_0 + \gamma z_k \psi_{j,k}(x) \in [c_0/2,3c_0/2]$ for every $z_k$. This implies that 
$\frac{d\p_{z^{\oplus k}}}{d\p_z}(x) = \frac{c_0 - \gamma z_k \psi_{j,k}(x)}{c_0 + \gamma z_k \psi_{j,k}(x)} \leq 3$. Thus,~\cref{assn:2} holds with $\lambda = 3$ (importantly, this is a constant). We now compute an upper bound on $\alpha^2 \eqdef \E_{\p_z} \brac{ \paren{ \frac{\gamma z_k \psi_{j,k}(X)}{c_0 + \gamma z_k \psi_{j,k}(X)} }^2 }$.
\begin{align*}
	\E_{\p_z} \brac{ \paren{ \frac{2 \gamma z_k \psi_{j,k}(X)}{c_0 + \gamma z_k \psi_{j,k}(X)} }^2 } 
	&= 4 \gamma^2 \int_{{\rm supp}(\psi_{j,k})} \frac{\psi_{j,k}(x)^2 (c_0 + \gamma z_k \psi_{j,k}(x))}{(c_0 + \gamma z_k \psi_{j,k}(x))^2}  dx \\
	&= 4 \gamma^2 \int_{{\rm supp}(\psi_{j,k})} \frac{\psi_{j,k}(x)^2}{c_0 + \gamma z_k \psi_{j,k}(x)} dx \\
	&\leq 2 \gamma^2 c_0 \int_{{\rm supp}(\psi_{j,k})} \psi_{j,k}(x)^2 dx \tag{as $c_0 + \gamma z_k \psi_{j,k}(x) \geq c_0/2$} \\
	&\leq 2 \gamma^2 c_0 \times (2^{j/2}\norm{\psi}_\infty)^2  \times {\rm length}({\rm supp}(\psi_{j,k})) \\
	&\leq 2 \gamma^2 c_0 \times (2^{j/2}\norm{\psi}_\infty)^2  \times \frac{C''}{2^j} \tag{for a constant $C''>0$} \\
	&= C' \gamma^2 \tag{for a constant $C'>0$} \\
	&= C 2^{-j(2s+1)} \tag{for a constant $C>0$}.
\end{align*}
Thus, using~\cref{thm:AcharyaCT20}, we get
\begin{equation} \label{eqn:lb1.0}
	\paren{ \frac{1}{2^j}\sum_{k \in \cI_j} \totalvardist{\p^{Y^n}_{-k}}{\p^{Y^n}_{+k}} }^2 \lesssim (n2^\numbits) 2^{-2j(s+1)}.
\end{equation}

\paragraph{Lower bound on average discrepancy.} To lower bound the average discrepancy, we will use the idea described in~\cref{sec:divlowbound}. Consider a communication-constrained density estimation algorithm (possibly interactive) that outputs $\hat{f}$ satisfying $\sup_{f \in \cB(p,q,s)} \bE{f}{\big\| \hat f-f \big\|_r^r} \leq \dst^r$. Using this density estimator, we estimate $\hat{Z}$ as
\[
\hat{Z} = \underset{z}{\rm argmin} \norm{f_z - \hat{f}}_r.
\] 
Then
\begin{equation} \label{eqn:lb1.1}
	\begin{aligned}
		\E_{\p_z} \brac{\norm{f_z - f_{\hat{Z}}}_r^r } 
		&\leq 2^{r-1} \paren{ \E_{\p_z} \brac{\norm{f_z - \hat{f}}_r^r } + \E_{\p_z} \brac{\norm{f_z - f_{\hat{Z}}}_r^r } }
		\leq 2^r \dst^r.
	\end{aligned}
\end{equation}
Now, for $z \neq z'$, we have
\begin{align*}
	\norm{f_z - f_{z'}}_r^r
	&= \int_{0}^{1} \abs{f_z(x) - f_{z'}(x)}^rdx \\
	&= \gamma^r \int_{0}^{1} \abs{\sum_{k \in \cI_j} \psi_{j,k}(x) \indic{z_k \neq z'_k}}^r dx \\
	&= \gamma^r \int_{0}^{1} \sum_{k \in \cI_j} \abs{\psi_{j,k}(x)}^r \indic{z_k \neq z'_k} dx \tag{since the $\psi_{j,k}$'s have disjoint supports} \\
	&= \gamma^r  \sum_{k \in \cI_j} \int_{{\rm supp}(\psi_{j,k})} \abs{\psi_{j,k}(x)}^r \indic{z_k \neq z'_k} dx \\
	&= \gamma^r (2^{j/2}\norm{\psi}_\infty)^r \frac{C''}{2^j} \sum_{k \in \cI_j} \indic{z_k \neq z'_k} \tag{for a constant $C''>0$}\\
	&= C' \gamma^r 2^{j(r/2-1)} \sum_{k \in \cI_j} \indic{z_k \neq z'_k} \tag{for a constant $C'>0$}\\
	&= C 2^{-j(rs+1)} \sum_{k \in \cI_j} \indic{z_k \neq z'_k} \tag{for a constant $C>0$}.
\end{align*}
which gives that, for an estimator $\hat{Z}$,
\[
\E_{\p_z} \brac{ \norm{f_z - f_{\hat{Z}}}_r^r } = C 2^{-j(rs+1)} \sum_{k \in \cI_j} \Pr\set{Z_k \neq \hat{Z}_k}.
\]
Combining this with \eqref{eqn:lb1.1}, we get
\begin{equation} \label{eqn:lb1.2}
	\sum_{k \in \cI_j} \Pr\set{\hat{Z}_k \neq Z_k} \lesssim \dst^r 2^{j(rs+1)}.
\end{equation}
Thus, substituting $d = 2^j$ and $\tau = 1/2$ in \eqref{eqn:dtvlb} (and ignoring multiplicative constants), we get
\begin{align*}
	\frac{1}{2^j}\sum_{k \in \cI_j}\totalvardist{\p^{Y^n}_{-k}}{\p^{Y^n}_{+k}} 
	&\gtrsim 1 - \frac{2}{2^j} \dst^r 2^{j(rs+1)} 
	\simeq 1 - \dst^r 2^{jrs}.
\end{align*}
Now, observe that $j$ is a free parameter that we can choose. If we choose $j$ such that
\begin{equation} \label{eqn:lb1.3}
	\dst^r 2^{jrs} \simeq 1
\end{equation}
then we get
\begin{equation*}
	\frac{1}{2^j}\sum_{k \in \cI_j}\totalvardist{\p^{Y^n}_{-k}}{\p^{Y^n}_{+k}} \gtrsim 1
\end{equation*}
or
\begin{equation} \label{eqn:lb1.4}
	\paren{ \frac{1}{2^j}\sum_{k \in \cI_j}\totalvardist{\p^{Y^n}_{-k}}{\p^{Y^n}_{+k}} }^2 \gtrsim 1
\end{equation}

\paragraph{Putting things together.} From \eqref{eqn:lb1.0} and \eqref{eqn:lb1.4}, we get that, for $j$ satisfying $\dst^r 2^{jrs} \simeq 1$,
\[
1 \lesssim (n2^\numbits) 2^{-2j(s+1)}.
\]
which gives
\[
2^j \lesssim (n2^\numbits)^{\frac{1}{2s+2}}.
\]
Using $\dst^r 2^{jrs} \simeq 1$, we finally get
\[
\dst^r \gtrsim (n2^\numbits)^{-\frac{rs}{2s+2}}
\]
which is our desired lower bound.

\subsection{Lower bound on $\cL_r^*(\ns, \numbits, p,q,s)$ for $r \geq (s+1)p$} \label{sec:lb2}
{\bf Construction.} The family of distributions $\cP_2$ that we will use to derive lower bound when $r \geq (s+1)p$ is not exactly the same as that in the unconstrained and in the LDP setting \cite{DonohoJKP96,HardleKPT12,ButuceaDKS20}; but, combined with the prior that we will choose on $Z$, it will essentially mimic that. \medskip

Let $g_0, \psi_{j,k}, \cI_j$ be as in~\cref{sec:lb1}.
For a given $z \in \set{-1,1}^d$ (where $d \eqdef \abs{\cI_j}$), define
\begin{equation}
	f_z \eqdef g_0 + \gamma \sum_{k \in \cI_j}(1+z_k) \psi_{j,k}.
\end{equation}
where we will choose $\gamma$ after we describe the prior on $Z$.
Finally, we define the family of distributions as
\begin{equation}
	\cP_2 = \set{\p_z : \p_z \text{ has density } f_z = g_0 + \gamma \sum_{k \in \cI_j}(1+z_k) \psi_{j,k}, \ z \in \set{-1,1}^d }.
\end{equation}

\paragraph{Prior on $Z$.} We assume a ``sparse'' prior on $Z \in \set{-1,1}^d$, defined as $Z_k \sim {\tt Rademacher}(1/d)$ independently for each $k \in \ibrac{d}$. We call it ``sparse'' because, with high probability, for $Z = (Z_1,\ldots,Z_d)$ sampled from this prior, the number of indices $k$ with $Z_k=1$ will be small (we will quantify this soon). Now, since $f_Z = g_0 + \gamma \sum_{k \in \cI_j}(1+Z_k) \psi_{j,k}$, this means that with high probability $1+Z_k = 0$ for a large number of $k$'s, and thus there will be only a few ``bumps'' in $f_Z$.

\paragraph{Choosing $\gamma$.} Define $\cG \subset \set{-1,1}^d$ as
\[
\cG \eqdef \set{ z \in \set{-1,1}^d: \sum_{k=1}^{d}\indic{z_k = 1} \leq 2j }.
\]
Then, by Bernstein's inequality
\begin{equation} \label{eqn:lb2bern}
	\Pr\set{Z \in \cG} \geq 1 - 4\cdot 2^{-2j}.
\end{equation}
We will choose $\gamma$ such that
\begin{itemize}
	\item[i.] $f_z(x) \geq c_0/2$ for every $x \in [a,b]$; as seen in~\cref{sec:lb1}, this condition is satisfied if $\gamma \leq (c_0/2\norm{\psi}_\infty) 2^{-j/2}$.
	\item[ii.] $\norm{f_z}_{pqs} \leq 1$ for every $z \in \cG$; argument similar to that in~\cref{sec:lb1} gives that $\norm{f_z}_{pqs} \leq 1$ for $z \in \cG$ if $\gamma \lesssim 2^{-j(s+1/2-1/p)} j^{-1/p}$.
\end{itemize}
Since $s > 1/p$, we get that for $j$ large enough ($j$ is a free parameter that we choose later), if $\gamma$ satisfies condition (ii), it automatically satisfies condition (i). Thus, we choose $\gamma = C 2^{-j(s+1/2-1/p)} j^{-1/p}$ for some constant $C$. \smallskip

Note that, for $z \notin \cG$, this choice of $\gamma$ still results in $f_z$ being a density function (since $\int \psi_{j,k}(x) dx = 0$), but it may be the case that $\norm{f_z}_{spq} > 1$.

\paragraph{Upper bound on average discrepancy.} To upperbound average discrepancy, we verify that $\cP_2$ satisfies the three assumptions described in~\cref{sec:divupbound}. For any $z \in \set{-1,1}^d$, $k \in \ibrac{d}$, we have
\begin{align*}
	\frac{dp_{z^{\oplus k}}}{dp_z} (x)
	&= \frac{g_0 + \gamma \sum_{k \in I_j}(1-z_k) \psi_{j,k}(x)}{g_0 + \gamma \sum_{k \in I_j}(1+z_k) \psi_{j,k}(x)} \\
	&= 1 - \frac{2\gamma z_k \psi_{j,k}(x)}{c_0 + \gamma \sum_{k \in I_j}z_k \psi_{j,k}(x)}
\end{align*}
which is same as what we had in~\cref{sec:lb1}. Similar arguments lead to the conclusion that~\cref{assn:1,assn:2,assn:3} are satisfied. Moreover, an upper bound on $\alpha^2 \eqdef \E_{\p_z} \brac{ \paren{ \frac{2\gamma z_i \psi_{j,k}(X)}{c_0 + \gamma z_i \psi_{j,k}(X)} }^2 }$ follows similarly (with different value of $\gamma$), and we get that
\[
\alpha^2 \leq C 2^{-2j(s+1/2-1/p)}j^{-2/p}. \tag{for a constant $C>0$}
\] 
Thus, using~\cref{thm:AcharyaCT20}, we get
\begin{equation} \label{eqn:lb2.0}
	\paren{ \frac{1}{2^j}\sum_{k \in \cI_j} \totalvardist{\p^{Y^n}_{-k}}{\p^{Y^n}_{+k}} }^2 \lesssim (n2^\numbits)2^{-2j(s+1-1/p)} j^{-2/p}.
\end{equation}

\paragraph{Lower bound on average discrepancy.} To lowerbound average discrepancy, we proceed as in~\cref{sec:lb1}. Consider a communication-constrained density estimation algorithm (possibly interactive) that outputs $\hat{f}$ satisfying $\sup_{f \in \cB(p,q,s)} \bE{f}{\big\| \hat f-f \big\|_r^r} \leq \dst^r$. Using this density estimator, we estimate $\hat{Z}$ as
\[
\hat{Z} = \underset{z}{\rm argmin} \norm{f_z - \hat{f}}_r.
\] 
Then, for $z \in \cG$,
\begin{equation} \label{eqn:lb2.1}
	\E_{\p_z} \brac{\norm{f_z - f_{\hat{Z}}}_r^r } \leq 2^r \dst^r.
\end{equation}
This only holds for $z \in \cG$ because the estimator's guarantee only holds if samples come from a density $f$ satisfying $\norm{f}_{spq}\leq 1$.
Now, for $z \neq z'$, plugging in the value of $\gamma$ in the calculation done in~\cref{sec:lb1}, we get
\begin{align*}
	\norm{f_z - f_{z'}}_r^r
	= C j^{-r/p}  2^{-j\paren{r(s-1/p)+1}} \sum_{k \in \cI_j} \indic{z_k \neq z'_k} \tag{for a constant $C>0$}
\end{align*}
which gives that, for any estimator $\hat{Z}$,
\[
\E_{\p_z} \brac{ \norm{f_z - f_{\hat{Z}}}_r^r } 
=C j^{-r/p}  2^{-j\paren{r(s-1/p)+1}} \sum_{k \in \cI_j} \Pr\set{\hat{Z}_k \neq Z_k}.
\]
Combining this with \eqref{eqn:lb2.1}, we get that
\begin{equation} \label{eqn:lb2.2a}
	\sum_{k \in \cI_j} \Pr\set{Z_k \neq \hat{Z}_k, Z \in \cG} \lesssim \dst^r j^{r/p} 2^{j\paren{r(s-1/p)+1}} .
\end{equation}
Thus,
\begin{align*}
	\sum_{k \in \cI_j} \Pr\set{\hat{Z}_k \neq Z_k}
	&= \sum_{k \in \cI_j} \Pr\set{Z_k \neq \hat{Z}_k, Z \in \cG} + \sum_{k \in \cI_j} \Pr\set{Z_k \neq \hat{Z}_k, Z \notin \cG} \\
	&\leq \sum_{k \in \cI_j} \Pr\set{Z_k \neq \hat{Z}_k, Z \in \cG} + \paren{ \sum_{k \in \cI_j} \Pr\set{\hat{Z}_k \neq Z_k| Z \notin \cG} } \Pr\set{Z \notin \cG} \\
	&\lesssim \dst^r j^{r/p} 2^{j\paren{r(s-1/p)+1}} + 2^j 2^{-2j} \tag{using \eqref{eqn:lb2.2a},\eqref{eqn:lb2bern}}.
\end{align*}
Thus, substituting $d = 2^j$ and $\tau = 1/d = 2^{-j}$ in \eqref{eqn:dtvlb} (and ignoring multiplicative constants), we get
\begin{align*}
	\frac{1}{2^j} \sum_{k \in \cI_j} \totalvardist{\p^{Y^n}_{-k}}{\p^{Y^n}_{+k}} 
	&\gtrsim 1 - \dst^r j^{r/p} 2^{j\paren{r(s-1/p)+1}} - 2^{-j} \\
	&\simeq 1 - \dst^r j^{r/p} 2^{jr\paren{s-1/p+1/r}}.
\end{align*}
Choosing $j$ such that
\begin{equation} \label{eqn:lb2.3}
	\dst^r 2^{jr(s-1/p+1/r)} j^{r/p} \simeq 1
\end{equation}
gives
\begin{equation*}
	\frac{1}{2^j} \sum_{k \in \cI_j} \totalvardist{\p^{Y^n}_{-k}}{\p^{Y^n}_{+k}} \gtrsim 1. 
\end{equation*}
or
\begin{equation} \label{eqn:lb2.4}
	\paren{ \frac{1}{2^j} \sum_{k \in \cI_j} \totalvardist{\p^{Y^n}_{-k}}{\p^{Y^n}_{+k}} }^2 \gtrsim 1. 
\end{equation}

\paragraph{Putting things together.} From \eqref{eqn:lb2.0} and \eqref{eqn:lb2.4}, we get, for any $j$ satisfying $\dst^r 2^{jr(s-1/p+1/r)} j^{r/p} \simeq 1$, that
$
1 \lesssim (n2^\numbits)2^{-2j(s+1-1/p)} j^{-2/p}.
$
This then yields
\begin{equation} \label{eqn:lb2.5}
	2^{2j(s+1-1/p)} j^{2/p} \lesssim n2^\numbits.
\end{equation}
To get a rough idea of the bound this will give, let us ignore $j^{2/p}$ to get,
\begin{equation} \label{eqn:lb2ruff1}
	2^j \lesssim \paren{n2^\numbits}^\frac{1}{2(s+1-1/p)}.
\end{equation}
Now, since $\dst^r 2^{jr(s-1/p+1/r)} j^{r/p} \simeq 1$, we get, roughly, (ignoring $j^{r/p}$)
\begin{equation} \label{eqn:lb2ruff2}
	2^j \simeq (1/\dst)^{\frac{1}{s-1/p+1/r}}.
\end{equation}
Combining \eqref{eqn:lb2ruff1}, \eqref{eqn:lb2ruff2}, we get that (up to logarithmic factors)
\[
\dst^r \gtrsim (n2^\numbits)^{-\frac{r(s-1/p+1/r)}{2(s+1-1/p)}}
\]
which is the desired bound, again up to logarithmic factors. We now show how a slightly more careful analysis lets us obtain the tight bound.

\paragraph{Bringing in log factors.} From \eqref{eqn:lb2.5}, we get that
\begin{equation} \label{eqn:lb2.6}
	2^j \lesssim \paren{n2^\numbits}^\frac{1}{2(s+1-1/p)} \paren{\log(n2^\numbits)}^{-\frac{2/p}{2(s+1-1/p)}}.
\end{equation}
Now, since $\dst^r 2^{jr(s-1/p+1/r)} j^{r/p} \simeq 1$, we get
\[
2^j \simeq (1/\dst)^{\frac{1}{s-1/p+1/r}} \paren{\log(1/\dst)}^{-\frac{1}{p(s-1/p+1/r)}}.
\]
Substituting this in \eqref{eqn:lb2.6}, we get
\[
(1/\dst)^{\frac{1}{s-1/p+1/r}} \paren{\log(1/\dst)}^{-\frac{1}{p(s-1/p+1/r)}} \lesssim
\paren{n2^\numbits}^\frac{1}{2(s+1-1/p)} \paren{\log(n2^\numbits)}^{-\frac{2/p}{2(s+1-1/p)}}
\]
or
\[
1/\dst \paren{\log(1/\dst)}^{-1/p} 
\lesssim 
\paren{n2^\numbits}^\frac{s-1/p+1/r}{2(s+1-1/p)} \paren{\log(n2^\numbits)}^{-(2/p)\frac{(s-1/p+1/r)}{2(s+1-1/p)}}.
\]
This implies that
\begin{align*}
	1/\dst &\lesssim \paren{n2^\numbits}^\frac{s-1/p+1/r}{2(s+1-1/p)} \paren{\log(n2^\numbits)}^{-(2/p)\frac{(s-1/p+1/r)}{2(s+1-1/p)}} 
	\paren{ \log\paren{ \paren{n2^\numbits}^\frac{s-1/p+1/r}{2(s+1-1/p)} \paren{\log(n2^\numbits)}^{-(2/p)\frac{(s-1/p+1/r)}{2(s+1-1/p)}} } }^{1/p} \\
	&\simeq \paren{n2^\numbits}^\frac{s-1/p+1/r}{2(s+1-1/p)} \paren{\log(n2^\numbits)}^{-(2/p)\frac{(s-1/p+1/r)}{2(s+1-1/p)} + \frac{1}{p}} \\
	&= \paren{n2^\numbits}^\frac{s-1/p+1/r}{2(s+1-1/p)} \paren{\log(n2^\numbits)}^{- \frac{1-1/r}{p(s-1/p+1)}}.
\end{align*}
Thus
\begin{equation}
	\dst^r \gtrsim 
	\paren{n2^\numbits}^\frac{r(s-1/p+1/r)}{2(s+1-1/p)} \paren{\log(n2^\numbits)}^{- \frac{r-1}{p(s-1/p+1)}}.
\end{equation}

\subsection{Concluding the proof of~\cref{theo:lb}}
Combining lower bounds from~\cref{sec:lb1,sec:lb2} with lower bounds in the classical setting \cite{DonohoJKP96} (where the rate transition happens at $r = (2s+1)p$), we get~\cref{theo:lb}. \qed

\end{document}